\documentclass{jams-l}

\usepackage{amssymb}
\usepackage{amsthm}
\usepackage[colorlinks,citecolor=red,pagebackref,hypertexnames=false]{hyperref}
\usepackage{color}
\usepackage{esint}
\usepackage{bbm}
\usepackage{graphicx}
\usepackage[notref,notcite]{} 
\usepackage{caption}

\def\bS{{\mathbb{S}}}

\DeclareMathOperator{\diam}{diam}
\def\Tan{\mathop\mathrm{Tan}} 					
\def\dist{\mathop\mathrm{dist}} 						

\newcommand{\ps}[1]{\left( #1 \right)}

\newcommand{\av}[1]{\left| #1 \right|}
\newcommand{\ip}[1]{\left\langle #1 \right\rangle}

\def\XXint#1#2#3{{\setbox0=\hbox{$#1{#2#3}{\int}$ }
\vcenter{\hbox{$#2#3$ }}\kern-.58\wd0}}

\newtheorem{theorem}{Theorem}[section]
\newtheorem{lemma}[theorem]{Lemma}

\theoremstyle{definition}
\newtheorem{definition}[theorem]{Definition}

\theoremstyle{remark}
\newtheorem{remark}[theorem]{Remark}

\numberwithin{equation}{section}

\newcommand{\R}{\mathbb{R}}
\newcommand{\N}{\mathbb{N}}
\newcommand{\Z}{\mathbb{Z}}
\newcommand{\C}{\mathbb{C}}

\newcommand{\hm}[1]{\mathcal{H}^{#1}}

\newcommand{\spt}{\mathrm{spt}}

\newcommand{\dr}{\frac{dr}{r}}

\newcommand{\mucubes}{\mathcal{D}_\mu}

\newcommand{\vr}{\varphi}
\newcommand{\chara}{\mathbbm{1}}
\newcommand{\wt}{\widetilde}

\newcommand{\eqt}[1]{\ensuremath{\stackrel{#1}{=}}}

\newcommand{\ds}{d_{\bS}}

\newcommand\blfootnote[1]{%
  \begingroup
  \renewcommand\thefootnote{}\footnote{#1}%
  \addtocounter{footnote}{-1}%
  \endgroup
}

\numberwithin{equation}{section}
\theoremstyle{plain}
\newtheorem{sublemma}[theorem]{Sublemma}
\newtheorem{corollary}[theorem]{Corollary}
\newtheorem{notation}[theorem]{Notation}

\title[Symmetric measures and singular integrals]{ $\Omega$-symmetric measures and related singular integrals}

\author{Michele Villa}
\address{School of Mathematics, University of Edinburgh, JCMB, Kings Buildings,
Mayfield Road, Edinburgh,
EH9 3JZ, Scotland.}
\email{m.villa-2 ``at" sms.ed.ac.uk}
\curraddr{}
\email{}
\thanks{}

\date{}

\dedicatory{}
\begin{document}

\maketitle
\begin{center}

\begin{minipage}[c][][r]{300pt}
\begin{small}
\textsc{Abstract.} Let $\bS \subset \C$ be the circle in the plane, and let $\Omega: \bS \to \bS$ be an odd bi-Lipschitz map with constant $1+\delta_\Omega$, where $\delta_\Omega\geq 0$ is small. Assume also that $\Omega$ is twice continuously differentiable. Motivated by a question raised by Mattila and Preiss in \cite{mattila-preiss}, we prove the following: if a Radon measure $\mu$ has positive lower density and finite upper density almost everywhere, and the limit 
\begin{align*}
    \lim_{\epsilon \downarrow 0} \int_{\C \setminus B(x,\epsilon)} \frac{\Omega\ps{(x-y)/|x-y|}}{|x-y|} \, d\mu(y)
\end{align*}
exists $\mu$-almost everywhere, then $\mu$ is $1$-rectifiable. 
To achieve this, we prove first that if an Ahlfors-David 1-regular measure $\mu$ is symmetric with respect to $\Omega$, that is, if
\begin{align*}
    \int_{B(x,r)} |x-y|\Omega\ps{\frac{x-y}{|x-y|}} \, d\mu(y) = 0 \mbox{ for all } x \in \spt(\mu) \mbox{ and } r>0,
\end{align*}
then $\mu$ is flat, or, in other words, there exists a constant $c>0$ and a line $L$ so that $\mu= c \hm{1}|_{L}$.
\end{small}
\end{minipage}
\end{center}
\blfootnote{\textup{2010} \textit{Mathematics Subject Classification}: \textup{28A75}, \textup{28A12} \textup{28A78}.

\textit{Key words and phrases.} Rectifiability, Calder\'{o}n-Zygmund theory, singular integrals, beta numbers, symmetric measures. 

M. Villa was supported by The Maxwell Institute Graduate School in Analysis and its
Applications, a Centre for Doctoral Training funded by the UK Engineering and Physical
Sciences Research Council (grant EP/L016508/01), the Scottish Funding Council, Heriot-Watt
University and the University of Edinburgh.
}
\tableofcontents
\section{Introduction}
While investigating for what kind of measures $\mu$ in $\C$ does the Cauchy transform exists $\mu$-almost everywhere (in the sense of principal values), Mattila (see \cite{mat95})  gave a complete characterisation of what he termed \textit{symmetric} measures; for a measure $\mu$, let us set
\begin{align}
    C(x,r) := \int_{B(x,r)} \ps{x-y} \, d\mu(y).
\end{align}
We say that a measure is \textit{symmetric} if it satisfies
\begin{align*}
C(x, r) = 0 \mbox{ for all } x \in \spt(\mu) \mbox{ and } r>0.  
\end{align*}
Here $\spt(\mu)$ denotes the support of $\mu$.
Mattila showed that any symmetric locally finite Borel measure on $\C$ is either discrete or continuous. In the latter case, it is either the $2$-dimensional Lebesgue measure (up to a multiplicative constant) or a countable sum of $1$-dimensional Hausdorff measures restricted to equidistant affine lines. Mattila needed such characterisation to understand the geometry of tangent measures of a measure $\mu$ for which the Cauchy transform exists $\mu$-almost everywhere (in the sense of principal values) - thus to understand the geometry of $\mu$ itself. 
Briefly after, Mattila and Preiss (see \cite{mattila-preiss}) generalised this to the higher dimensional equivalent. 
Shortly after, Huovinen in \cite{huovinen} proved a similar result for measures $\mu$ symmetric with respect to kernels of the type $K(z) = \frac{z^{k}}{|z|^{k+1}}$, where $k$ is an odd integer.
Measure symmetric with respect to more general kernels appeared  recently in the works of Jaye and Merch\'{a}n (see \cite{JM1} and \cite{JM2}), where the authors give some necessary and sufficient conditions for the principal value integral of the corresponding kernel to exist.

Let us now introduce the main result of this paper; after we will give a few remarks on its possible applications. Consider a measure $\mu$ in $\C$, and let $K$ be a measurable map $\C\setminus\{0\} \to \C\setminus\{0\}$ given by $K_{\Omega}(x)= K(x)= |x| \Omega\ps{\frac{x}{|x|}}$, where $\Omega : \bS \to \bS$. 
\begin{definition}
We say that $\mu$ is $\Omega$-symmetric if 
\begin{align}
    C_{\Omega,\mu}(x,r):= \frac{1}{r} \int_{B(x,r)} \frac{K(x-y)}{r} \, d\mu(y) = 0
\end{align}
for all $x \in \spt(\mu)\subset \C$ and $r>0$.
\end{definition}
\begin{remark}
One could similarly consider kernels $K$ given by $K(x)= \rho(|x|) \Omega(x/|x|)$. However Lemma \ref{lemma:SYM_borelequiv} below says that this would not give an effectively more general result. 
\end{remark}

\begin{definition}[Ahlfors regularity]
A measure $\mu$ on $\C$ is called Ahlfors 1-regular (or Ahlfors-David regular) with constant $C_0>0$ if for all $x\in \spt(\mu)$ and all $r>0$, we have that
\begin{align} \label{e:ADR}
   C_0^{-1} r \leq \mu\ps{B(x,r)} \leq C_0 r.
\end{align}

\end{definition}

Our aim here is to prove the following result.
\begin{theorem} \label{theorem:main}
Let $K: \C\setminus\{0\} \to \C$ be given by $K(x) = |x| \Omega\left( \frac{x}{|x|}\right)$, where $\Omega: \bS \to \bS$ is an odd, twice continuously differentiable map which is also bi-Lipschitz with constant $1+\delta_\Omega$ \textup{(}with respect to the geodesic distance on the 1-sphere\textup{)}. Let $\mu$ be an Ahlfors 1-regular, $\Omega$-symmetric measure in $\C$. Then, if $\delta_\Omega$ is sufficiently small \textup{(}smaller than some absolute constant\footnote{For example, having $\delta_\Omega \leq \frac{1}{20}$ will work.}\textup{)}, $\mu = c \mathcal{H}^1|_{L}$ for some line $L$. 
\end{theorem}

This result can be considered a first step towards answering a question posed by Mattila and Preiss in \cite{mattila-preiss}, Remark 4.4 (2): given a measure $\mu$ in $\R^d$ with positive $n$-dimensional density, the existence of the principal values for singular integrals with kernels such as $K(x)= \frac{\Omega(x/|x|)}{|x|^n}$ implies that the tangent measures $\nu$  of $\mu$ satisfy an $\Omega$-symmetricity condition:
\begin{align}
    \int_{B(x,r)} |x-y| \Omega((x-y)/|x-y|) \, d\nu(y) \mbox{ for } x \in \spt(\nu) \mbox{ and } r>0.
\end{align}
For what kernels does this imply rectifiability?
In the last section of this paper we prove the following application of Theorem \ref{theorem:main}.
\begin{corollary}  \label{cor:singular-integrals}
Let $K$ be as in the statement of Theorem \ref{theorem:main} and let $\mu$ be a Radon measure in the plane such that for $\mu$-almost all $x \in \C$ the following conditions hold. 
\begin{itemize}
    \item The lower density is positive and the upper density is finite, that is 
\begin{align}
     & \theta^{1,*}(\mu,x) := \limsup_{r \to 0} \frac{\mu(B(x,r))}{r} < + \infty; \label{e:up-density}\\
     & \theta_{*}^1(\mu, x) := \liminf_{r \to 0} \frac{\mu(B(x,r))}{r} >0. \label{e:low-density}
\end{align}
\item The principal value 
\begin{align} \label{e:pv}
    \lim_{\epsilon \downarrow 0} \int_{\C \setminus B(x,\epsilon)} \frac{K(x-y)}{|x-y|^2} \, d\mu(y) 
\end{align}
exists and is finite.
Then $\mu$ is $1$-rectifiable.
\end{itemize}
\end{corollary}

We give a couple of remarks on the assumptions that we make in Theoren \ref{theorem:main} and Corollary \ref{cor:singular-integrals}.
\begin{remark}
The assumption on the lower density is somewhat unsatisfactory. A natural question is therefore whether Corollary \ref{cor:singular-integrals} still holds without such an assumption. 
\end{remark}

\begin{remark}
The existence of principal values on rectifiable set has been known for a while under the assumption of having many derivatives of the kernel (see Corollary 1.6 and the remark below in \cite{mas}, and Chapter 20 (Theorem 20.15) in \cite{mattila}). 
More recently, however, Mas proved (see Corollary 1.6 in \cite{mas}) that if a measure on $\R^d$ is $n$-rectifiable, then the principal values exists for a large class of kernels, namely, for odd kernels with derivative bounds
$$
|\nabla^j K(x) | \lesssim \frac{1}{|x|^{n+j}}, \, \, j=0,1,2 \mbox{ and for all } x \in \R^d\setminus \{0\}. 
$$
In view of this, the assumption on the second derivative of $\Omega$ in Theorem \ref{theorem:main} looks natural. 
\end{remark}

\begin{remark}
The assumption of Ahlfors regularity in Theorem \ref{theorem:main} is quite strong. A more significant step in answering the question of Mattila and Preiss would be to get rid of such an assumption.
\end{remark}

A second application will be presented in an upcoming paper  \cite{villa19a}, where we give a new characterisation of uniform rectifiability. In this context, the assumption of Ahlfors regularity is natural. We refer the reader to \cite{villa19a} for more details. 

\subsection{Outline}
The core of this note is showing that if a measure $\mu$ is $\Omega$-symmetric then its support lies in some line $L$. To do so, we will use techniques from \cite{tolsa2007} and \cite{tolsa2008} to obtain a bound for the sum of the Jones $\beta$ over all cubes contained in some top cube $Q_0$. Hence, we will show that such sum goes to zero as we increase the size of $Q_0$ to infinity. Such a strategy is a modification of some ideas in \cite{mattila-preiss}, Section 6. 

\subsection{Acknowledgements}
I would like to warmly thank Jonas Azzam for his patience and kind support. I also thank Xavier Tolsa and Joan Mateu for several useful conversations. I am grateful to the anonymous referee for several useful remarks which greatly improved the exposition.
\section{Preliminaries}

\subsection{Notation}
We gather here some notation and some results which will be used later on.
We write $a \lesssim b$ if there exists a constant $C$ such that $a \leq Cb$. By $a \sim b$ we mean $a \lesssim b \lesssim a$.

For sets $A,B \subset \R^n$, we let
\begin{align*}
    \dist(A,B) := \inf_{a\in A, b \in B} |a-b|.
\end{align*}
For a point $x \in \R^n$ and a subset $A \subset \R^n$, 
\begin{align*}
    \dist(x, A):= \dist(\{x\}, A)= \inf_{a\in A} \dist(x,a).
\end{align*}
We write 
\begin{align*}
    B(x, t) := \{y \in \R^n \, |\,|x-y|<t\}.
\end{align*}

We will call $\bS$ the $1$-sphere in $\C$; for
two points $x, y \in \bS$, 
\begin{align} \label{e:ds}
    \mbox{we denote by } \ds \mbox{ the geodesic distance on } \bS.
\end{align}
 
\begin{remark} \label{r:dot}
Throughout this note, the notation $$M \cdot x$$ will be used to indicate a matrix $M$ acting on a vector $x$ | this will never be used to denote the dot product (which we will denote with the standard $\left< \cdot , \cdot \right>$).
\end{remark} 

Let $A \subset \C$ and $0< \delta \leq \infty$. Set
\begin{align*}
    \mathcal{H}^1_\delta (A) := \inf \left\{\sum \diam (A_i) \, |\, A \subset \cup_i A_i \, \mbox{and } \diam(A_i) \leq \delta \right\}.
\end{align*}
The 1-dimensional Hausdorff measure of $A$ is then defined by
\begin{align*}
\mathcal{H}^1(A) := \lim_{\delta \to 0} \mathcal{H}^1_\delta(A).
\end{align*}

\subsection{Jones $\beta$ numbers}
Let $\mu$ be a 1-Ahlfors regular measure (see \eqref{e:ADR}). We define
\begin{align} \label{eq:01_beta}
\beta_{\mu, p} (x,t):=\inf_P \left(\frac{1}{t} \int_{B(x,t)} \left(\frac{\dist(y,
P)}{t}\right)^p \, d\mu(y) \right)^{\frac{1}{p}},
\end{align}
were the infimum is taken over all affine lines $P$. 
This quantity measures how far the support of $\mu$ is from being a line. Its relation with problems involving singular integrals and rectifiability has a fairly long history. Variants of such coefficients were firstly introduced by Jones in \cite{jones} (hence Jones-$\beta$-numbers) while working on the Analyst's traveling salesman problem in the plane (see also \cite{oki} and \cite{sch} application of these coefficients in the euclidean space and in Hilbert space); shortly after  David and Semmes introduced a variant of these coefficients (and this is the variant which we will use) to develop their theory of uniformly rectifiable sets (see for example \cite{singularintegrals} and \cite{analysis}); the Jones coefficients have been extensively used in recent years within Geometric Measure Theory; see for example the Reifenberg-type parameterisation results by David and Toro in \cite{david-toro} (see also Ghinassi in \cite{ghinassi}) and by Edelen, Naber and Valtorta \cite{env}; the series of Tolsa and Azzam and Tolsa \cite{tolsa2015}, \cite{tolsa2017} and \cite{azzam-tolsa}; the series of Badger and Schul \cite{bs15}, \cite{bs16}, \cite{bs17}.  

\subsection{Intrinsic cubes with small boundaries}
The following construction, due to David in \cite{wavelets}, provides us with a dyadic decomposition of the support of an AD-regular measure. Such construction has been extended by Christ in \cite{christ} to spaces of homogeneous type and further refined by Hyt\"{o}nen and Martikainen in \cite{hytonen}. Here is the construction. 

\begin{theorem} \label{theorem:ADcubes}
Let $\mu$ be an $n$-AD regular measure in $\R^d$. There exists a collection $\mucubes$ of subsets $Q \subset \spt(\mu)$ with the following properties. 
\begin{enumerate}
    \item We have
    \begin{align*}
        \mucubes = \bigcup_{j \in \Z} \mucubes^j,
    \end{align*}
    where $\mucubes^j$ can be thought as the collection of cubes of sidelength $2^{-j}$. 
    \item For each $j \in \Z$, 
    \begin{align*}
        \spt(\mu) = \bigcup_{Q \in \mucubes^j} Q.
    \end{align*}
    \item If $j \leq i$, $Q \in \mucubes^j$, $Q' \in \mucubes^i$, then either $Q \subset Q'$ or else $Q \cap Q' = \emptyset$.
    \item If $j \in \Z$ and $Q \in \mucubes^j$, then there exists a constant $C_0\geq 1$ so that
    \begin{align*}
       &  C_0^{-1}2^{-j} \leq \diam(Q) \leq C_0 2^{-j}, \mbox{ and } \\
        & C_0^{-1}2^{-jn} \leq \mu(Q) \leq C_0 2^{-jn}.
    \end{align*}
    \item If $j \in \Z$, $Q \in \mucubes^j$ and $0<\tau<1$, then 
    \begin{align*}
        \mu \left(\left\{ x \in Q \, |\, \dist(x, \spt(\mu) \setminus Q) \right\}\right) \leq C \tau^{\frac{1}{C}} 2^{-nj}.
    \end{align*}
\end{enumerate}
\end{theorem}
For a proof of this, see Appendix 1 in \cite{wavelets}.

\begin{notation}
For $Q \in \mucubes^j$, we set
\begin{align}
    \ell(Q) := 2^{-j}. \label{eq:cubelength}
\end{align}
We will denote the center of $Q$ by $z_Q$. Furthermore, we set
\begin{align*}
    B_Q := B(z_Q, 3 \diam(Q)).
\end{align*}
For a cube $Q \in \mucubes$, we set
\begin{align} \label{e:betaQ}
    \beta_{\mu,p}(Q) := \beta_{\mu,p}(B_Q) = \beta_{\mu,p}(z_Q, 3\diam(Q)).
\end{align}
\end{notation}

\subsection{Balanced points}
The following Lemma will be very useful; it holds for a general $n$-ADR measure in $\R^d$ but we state it taylored to our context.
\begin{lemma}[{\cite{singularintegrals}, Lemma 5.8}]\label{lemma:SYM_balancedcubes}
Let $\mu$ be an Ahlfors 1-regular measure with constant $C_0$ in $\R^d$. There exists a constant $\eta= \eta(C_0)<1$ so that for each cube $Q \in \mucubes$ we can find two points $x_0, x_1 \in Q$ so that
\begin{align}
    |x_1-x_0| \geq \eta \ell(Q)
\end{align}
\end{lemma}
We may refer to $x_0$ and $x_1$ as the `balanced points', and the unique line they span as the `balanced line';
\begin{align} \label{e:balanced-line}
    \mbox{we will denote such line by $L_Q$.}
\end{align}
Note that the same lemma holds for any ball $B(x, r)$ centered on $\spt(\mu)$; in this instance we will denote the balanced line as $L_{x, r}$.

\subsection{Preliminaries on $C_\Omega$}
The following Lemma can be found in \cite{mattila}, Theorem 20.6, , \cite{mat95}, Theorem 3.2, or in \cite{huovinen}, Lemma 3.5 in slightly different form. 
\begin{lemma}[\cite{huovinen}, Lemma 3.5] \label{lemma:SYM_borelequiv}
The following two conditions are equivalent.
\begin{enumerate}
    \item $\int_{B(x,r)} K(x-y) \, d\mu(y) = 0$ for  $\mu$-a.e.  $x \in \C$ and  $r>0$.
    \item $\int_{
    } K(x-y) \phi(|x-y|) \, d\mu(y) = 0$ for $\mu$-a.e. $x \in \C$, $r>0$ and for all bounded Borel functions $\phi: \R_+ \to \R$ such that $\lim_{r \to \infty}\frac{|\phi(r)|}{\mu(B(0,r))}= 0$. 
\end{enumerate}
\end{lemma}
\begin{remark} \label{remark:SYM_borelequiv}
We will use 
Lemma \ref{lemma:SYM_borelequiv} several times below. For later use, set
\begin{align} \label{e:C-smooth}
    C_{\Omega, \phi}(x,r):= \frac{1}{r} \int \frac{K(x-y)}{r} \phi \ps{\av{\frac{x-y}{r}}^2} \, d\mu(y),
\end{align}
where $\phi$ is a bounded Borel function as in the statement of the lemma.
Let us anticipate that, below, we will apply this definition when
 $\phi$ is a smooth cut off of the annulus $A(0, 1/2, 1)$, for example. We will specify the exact form if $\phi$ in the relevant sections. 
\end{remark}

\section{Some preparatory lemmas}\label{s:split}
Let us fix a cube $Q \in \mucubes$ and a constant $A >1$ which will be bounded above later. In the following, we will assume that 
\begin{align} \label{eq:SYM:r}
    r \in [A \ell(A), 2 A \ell(Q)].
\end{align}
\begin{remark} \label{rem:line_to_line}
An obvious fact, which we will use over and over, is that the map $K(x)= |x| \Omega(x)$, because it is anti-symmetric, maps a line trough the origin to another line through the origin; this is false $d$-dimensional affine planes in $\R^n$, say, and it is one of the difficulties to go through to generalise the arguments which we use here beyond the plane.
\end{remark}

\begin{remark} \label{r:Omega}
From now on, we take $\Omega$ so that it satisfies all the hypotheses of Theorem \ref{theorem:main}, that is, $\Omega$ is an odd, $\mathcal{C}^2$ map from the circle to the circle, which is moreover bi-Lipschitz with constant $1+\delta_\Omega$, and $\delta_\Omega$ is sufficiently small (for definiteness, we choose now 
\begin{align}\label{e:delta-Omega}
    \delta_\Omega \leq \frac{1}{20}.
\end{align}
Whenever we write $K(x)$ we mean $K(x)=|x| \Omega(x/|x|)$, with $\Omega$ as above. 
\end{remark}

We start with an elementary lemma.
\begin{lemma} \label{lemma:path-lift}
Let $\Omega : \bS \to \bS$ be a bi-Lipschitz function of the sphere. Then there exists a unique bi-Lipschitz function $\omega: [0,2\pi] \to [0, 2\pi]$ such that if $t \in [0,2\pi]$,
\begin{align}
    \Omega(\exp(it)) = \exp(i\omega(t)).
\end{align}
\end{lemma}
\begin{proof}
It is well known that there exists one such continuous  function $\omega$ whenever $\Omega$ is continuous. We show that $\omega$ is also bi-Lipschitz. For $t \neq s$ points in $[0,2\pi]$, we have that (recall the definition of $\ds$ in \eqref{e:ds}),
\begin{align*}
    & |t-s| = d_{\bS}(\exp(it), \exp(is)) \\
    &\geq (1+\delta_\Omega) d_{\bS}(\Omega(\exp(it)), \Omega(\exp(is)))\\
    & = (1+\delta_\Omega) d_{\bS}(\exp(i\omega(t)), \exp(i\omega(s))) = (1+\delta_\Omega)|\omega(t)-\omega(s)|. 
\end{align*}
The reverse inequality is shown in the same way.
\end{proof}

\begin{remark} \label{r:omega(0)}
We may (and we will) assume that $\omega(0)=0$, where $\omega$ is the map given in Lemma \ref{lemma:path-lift}. Indeed, if $\omega(0)=t_0$, then one can reduce to the case $t_0=0$ as follows: let $R$ be the rotation so that $R (e^{it_0}) = 1$. It is immediate to see that if $\mu$ is $\Omega$-symmetric, then it is also symmetric with respect to the kernel $\wt \Omega = R \circ \Omega$. Clearly, the corresponding map $\wt \omega$ has $\wt \omega(0)=0$. 
\end{remark}

In the next remark, we summarise some facts that will turn out to be useful later on. We also set some notation. 
\begin{remark} \label{remark:BT_remdiff}
Consider any $v \in \C$.
If by the product rule we write
\begin{align} \label{eq:SYM_genrem1}
    DK(y) \cdot v = \left<(D|y|), v \right> \Omega(y) + |y| D \Omega(y) \, \cdot v,
\end{align}
we see the  following.
\begin{itemize}
    \item First, 
    \begin{align}\label{e:av-der}
        (D|y|)\Omega(y) \cdot v = \left< \frac{y}{|y|}, v\right> \Omega(y).
    \end{align}
    Now if we split $v$ into the component parallel to the span of $y$ and its orthogonal complement (let us denote them by $v^{\parallel (y)}$ and $v^{\perp (y)}$, respectively | we may just write $v^\parallel$ and $v^\perp$ when the $y$-dependence is obvious), then clearly 
    \begin{align*}
        (D|y|) \Omega \cdot v= \left< \frac{y}{|y|}, v^\parallel\right> \Omega(y).
    \end{align*}
    \item Second, let us consider the second term on the right hand side of \eqref{eq:SYM_genrem1}; the differential $D\Omega$ is a map 
\begin{align*}
    D \Omega(y/|y|) = d_{y/|y|} \Omega \circ d_y (\cdot /|\cdot|) : T_y\C \to T_{y/|y|} \bS \to T_{\Omega(y/|y|)} \bS.
\end{align*}
where $d_{y/|y|}\Omega: T_{y/|y|} \bS \to T_{\Omega(y/|y|)} \bS$ is the differential of $\Omega$ at $y/|y|$, $d_{y} (\cdot/|\cdot|)$ is the differential of the map $y \mapsto y/|y|$ at $y$, $T_{y}\C$ is the tangent plane of $\C$ at $y$, $T_{y/|y|}\bS$ (resp. $T_{\Omega(y/|y|)}$) is the tangent plane of $\bS$ at $y/|y|$ (resp. at $\Omega(y/|y|)$).
Note that 
\begin{align*}
    d_y (\cdot/|\cdot|) \cdot v & = \frac{v}{|y|} - \left< v, y\right> \frac{y}{|y|^3} 
    = \frac{1}{|y|}\ps{v- \ip{v,\frac{y}{|y|}}\frac{y}{|y|}} \\
    & = \frac{\Pi_{(\mathrm{span}(y))^\perp} (v)}{|y|};
\end{align*}
 therefore we have that 
 \begin{align*}
     D\Omega(y) \cdot v = d_{\frac{y}{|y|}} \Omega\cdot  \frac{1}{|y|} v^\perp.
 \end{align*}
Now,
\begin{align} \label{e:y-perp}
    & \text{we choose $\hat y_\perp$ to be the vector $y/|y|$ } \nonumber \\
    & \mbox{rotated by $90$ degrees counter-clockwise.} 
\end{align}
Then we can write $v^\perp$ as 
\begin{align}
    v^\perp = \left< v, \hat y_\perp\right> \hat y_\perp,
\end{align}
and hence
\begin{align} \label{e:der-sphere}
   D \Omega(y) \cdot v =  \frac{1}{|y|}\left< v, \hat y_\perp\right> d_{\frac{y}{|y|}}\Omega \cdot \hat y_\perp.
\end{align}
\end{itemize}
\end{remark}

\begin{notation}
We set 
\begin{align} \label{e:yhat-ex}
    \hat y := y /|y| \mbox{ and } e_{x_1} = \frac{x_1}{|x_1|}.
\end{align}
\end{notation}
We choose this notation to emphasise that $e_z$ (as defined in \eqref{e:e_z}) and $e_{x_1}$ together span $\C$; they are in fact an orthonormal basis. On the other hand, we defined $y/|y|$ as $\hat y$ to emphasise the fact that it is an element of $\bS$. 

For future reference, we summarise these remarks in the following Sublemma; it follows immediately from \eqref{eq:SYM_genrem1}, \eqref{e:av-der} and \eqref{e:der-sphere}.
\begin{sublemma} \label{sl:DK-expression}
Keep the notation as above. Then
\begin{align}\label{e:DK-expression}
    DK(y) \cdot v = \left< \hat y, v^\parallel\right> \Omega(y) + \left< v, \hat y_\perp\right> d_{\hat y}\Omega \cdot \hat y_\perp.
\end{align}
\end{sublemma}

The following technical lemma will be very useful to control how dot products behave under the map $K$. 
\begin{lemma}\label{l:dot-prod-sign}
Let $K$ be the map as in Remark \ref{r:Omega}, let  $y \in \C$ and let $L$ be a line through the origin. Denote by $\tilde \nu$ the normal unit vector to $L$ and by $\nu$ the normal unit vector to $K(L)$. If 
\begin{align}\label{e:dot1}
    \left< \hat y, \wt \nu \right> \geq \frac{1}{10}, 
\end{align}
then 
\begin{align} \label{e:dot1b}
    \left< \Omega(y), \nu \right> \geq \frac{1}{20}.
\end{align}
Moreover, if
\begin{align} \label{e:dot2}
    \left< \hat y, e_L \right> \geq \frac{1}{10}, 
\end{align}
where $\mathrm{span}(e_L)=L$, then also
\begin{align} \label{e:dot2b}
    \left< \Omega(y), \Omega(e_L) \right> \geq \frac{1}{20}.
\end{align}
\end{lemma}
\begin{figure}
\centering
\begin{minipage}{.5\textwidth}
  \centering
  \includegraphics[scale=0.5]{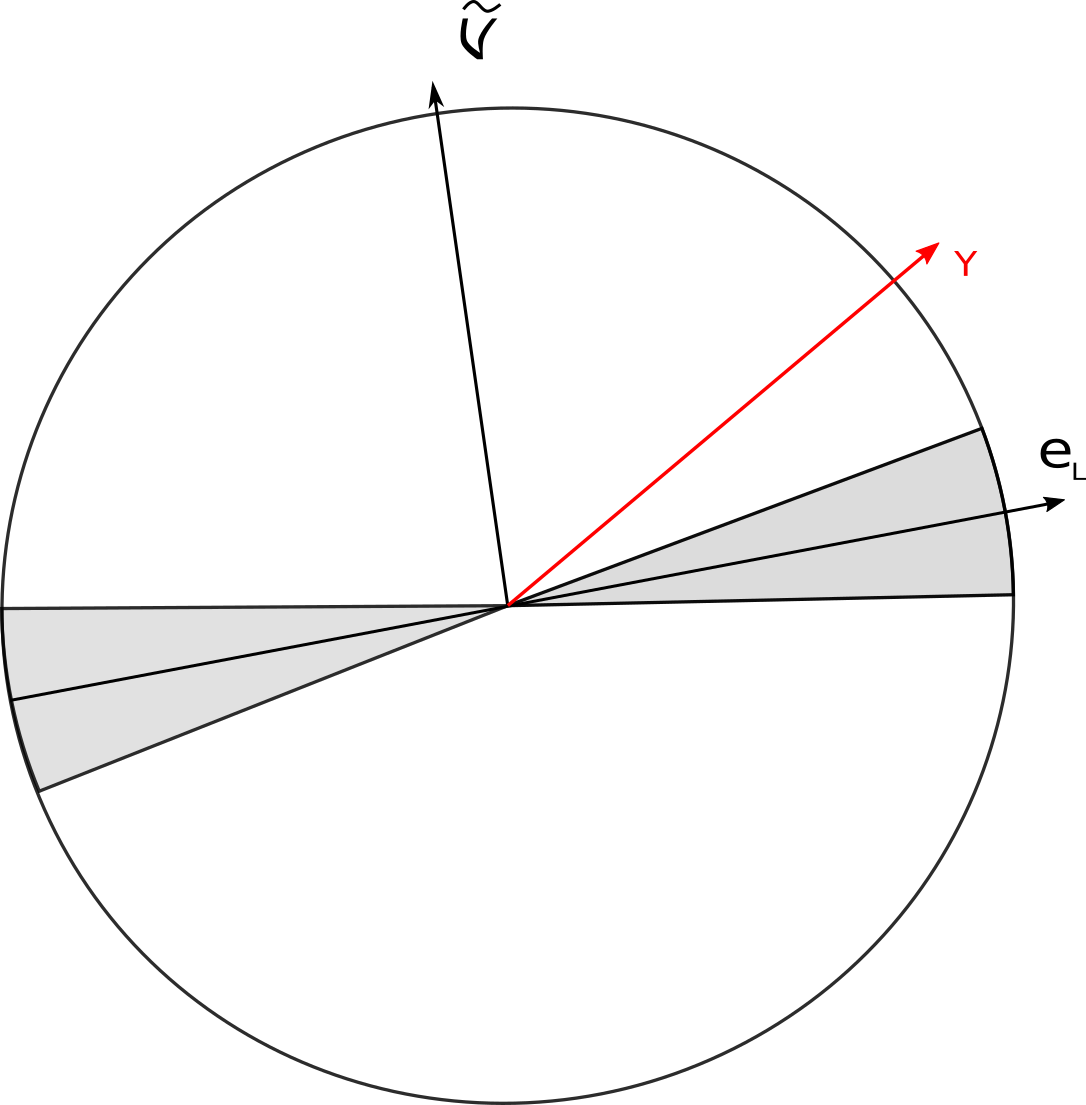}
  \captionof{figure}{When \eqref{e:dot1} holds, $y$ cannot lie in the shaded area.}
  \label{f:1}
\end{minipage}%
\begin{minipage}{.5\textwidth}
  \centering
  \includegraphics[scale=0.5]{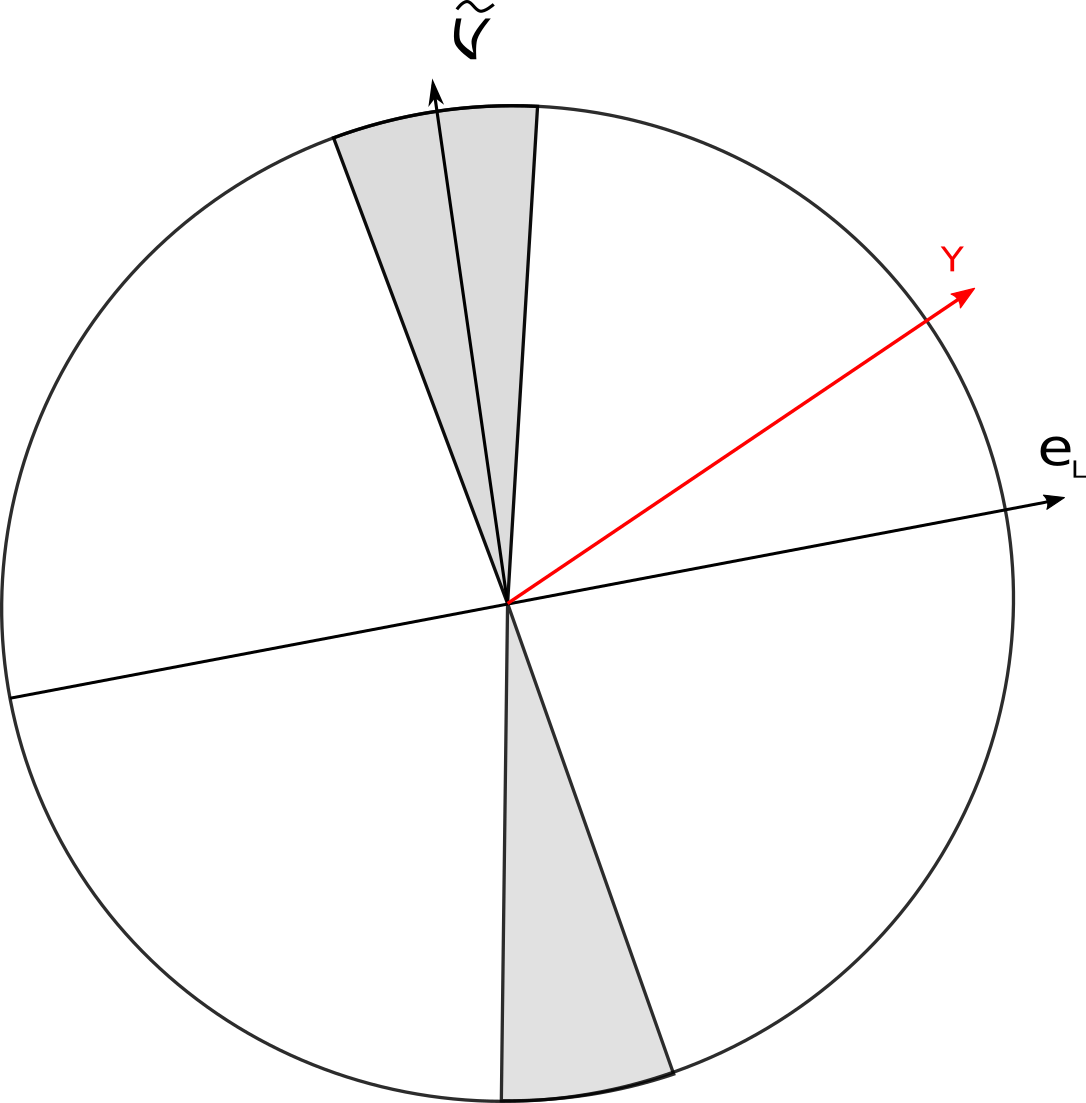}
  \captionof{figure}{This is the case when \eqref{e:dot2} holds}
  \label{f:2}
\end{minipage}
\end{figure}
\begin{proof}
We first show that \eqref{e:dot1} implies \eqref{e:dot1b}. Denote by $\theta_{\wt \nu}$ the angle that the unit vector $\wt \nu$ makes with the positive real axis and by $\theta_L$ the same for $e_L$. Let $\C^+$ to be the half plane `above' $L$ and $\C$ the upper half plane `below' $L$. We will always take the normal $\wt \nu$ which lies in $\C^+$.

   Note that \eqref{e:dot1} forces $y$ to lie in $\C^+$, and therefore $|\theta_y - \theta_{\wt \nu}| \leq \frac{\pi}{2}$. Moreover, because $\cos$ is even, we can assume that $y$ belongs to the first quadrant $Q_1$\footnote{By first quadrant we simply mean the set of points with positive $\wt \nu$ and $e_L$ coordinates.} (see Figure \ref{f:1}). 
    We write
\begin{align*}
    \left< \Omega(y), \nu \right> = \cos( \omega(\theta_y) - \theta_{\wt \nu} ) = \cos(\omega(\theta_y) - \omega(\theta_L) - \pi/2) = \sin(|\omega(\theta_y) - \omega(\theta_L)|). 
\end{align*}
Now, since $\Omega$ is bi-Lipschitz with constant $1+\delta_\Omega$, we have that $\frac{|\theta_y - \theta_L|}{1+ \delta_\Omega} \leq |\omega(\theta_y) - \omega(\theta_L)| \leq (1+\delta_\Omega) |\theta_y - \theta_L|$. 
Set 
\begin{align*}
    \alpha:=|\theta_y- \theta_L|.
\end{align*}
We have two possibilities; if 
\begin{align} \label{e:case1-a}
    \sin(|\omega(\theta_y)-\omega(\theta_L)|) \geq \sin((1+\delta_\Omega)\alpha),
\end{align}
we compute as follows\footnote{This is the case where $\Omega(y)$ lies very close to $\nu$, and increasing the angle makes $\sin$ smaller.}. By Taylor's theorem, we have
\begin{align} \label{e:case1-a-1}
    \sin((1+ \delta_\Omega) \sin(\alpha)) = \sin(\alpha) + \cos(\xi)\delta_\Omega \alpha,
\end{align}
where $\xi \in [ \alpha, (1+\delta_\Omega)\alpha]$ (in particular $\cos(\xi)$ could be negative). 
Since we assumed that $\delta_\Omega \leq 1/20$ (see Remark \ref{r:Omega}), we have that \eqref{e:case1-a-1} is bounded below by $\sin(\alpha)- \frac{\alpha}{20}$; now it suffices to notice that the inequality $\sin(\alpha) - \alpha/20 \geq \frac{1}{2} \sin(\alpha)$ is satisfied since $0< \alpha < \pi/2$. 
Moreover, $\sin(\alpha)= \left<\hat y, \wt \nu\right>$; this proves that $\left< \Omega(y), \nu \right> \geq \frac{1}{20}$ in this case (i.e. \eqref{e:case1-a}).

 The computation for
\begin{align}\label{e:case1-a-2}
    \sin(|\omega(\theta_y) - \omega(\theta_L)|) \geq \sin( (1+\delta_\Omega)^{-1} \alpha)
\end{align}
is similar. First, we expand the right hand side with Taylor's theorem to obtain the expression
\begin{align*}
    \sin(\alpha) - \cos( \xi) \frac{\delta_\Omega\alpha}{1+ \delta_\Omega}.
\end{align*}
We want $\delta_\Omega/(1+\delta_\Omega) \leq \frac{1}{2}\sin(\alpha)$; this is implied by $\alpha/20 \leq \frac{1}{2} \sin(\alpha)$, which holds for $0 < \alpha \leq \pi/2$. This gives \eqref{e:dot1b} also for \eqref{e:case1-a-2}. We have proven that \eqref{e:dot1} implies \eqref{e:dot1b}. 

Let us now show that \eqref{e:dot2} implies \eqref{e:dot2b}\footnote{We should be careful enough here to re-define $\omega$ so that it is defined on $(-\pi, \pi]$; simply put $\wt \omega(t) := \omega(t+\pi)-\pi$. Recall also Remark \ref{r:omega(0)}.}. The proof is very similar to the one just given, but let us include it for the sake of completeness. With the same notation as above, we write
\begin{align*}
    \left< \Omega(y), \Omega(e_L) \right> = \cos( |\omega(\theta_y) - \omega(\theta_L)|).
\end{align*}
Put $\alpha := |\theta_y - \theta_L|$.
This time, we can only have  $\cos(|\omega(\theta_y) - \omega(\theta_L)|) \geq \cos((1+ \delta_\Omega)\alpha)$, since $\cos$ can only decrease whenever we make the angle larger. We compute as above: by Taylor's theorem, we have 
\begin{align*}
    \cos((1+\delta_\Omega)\alpha) = \cos(\alpha) - \sin(\xi) \delta_\Omega \alpha. 
\end{align*}
One can then check that the inequality $\sin(\xi) \delta_\Omega \alpha \leq \frac{1}{2} \cos(\alpha)$ is satisfied for the relevant range of $\alpha$. We can then conclude that $\left<\Omega(y), \Omega(e_L)\right> \geq \frac{1}{2} \left< \hat y, e_L \right> \geq \frac{1}{20}$. The lemma then follows. 
\end{proof}
\begin{remark}\label{r:C+}
Lemma \ref{l:dot-prod-sign} holds similarly when $-1\leq \left< \hat y, \wt \nu \right> \leq -\frac{1}{10}$ or $-1 \leq \left< \hat y, e_L \right> \leq - \frac{1}{10}$. It suffices to carry out the computations as in the proof of the lemma with $-\wt \nu, -\nu$ and $-e_L$, respectively. 
\end{remark}

The lemma above says that we can control both the sign and the size of dot products in the image whenever the ones in the domain are sufficiently large (or small). The next lemma says that when they are small, we still have some control in terms of absolute values. 
\begin{lemma}\label{l:dot-abs}
    Keep the notation as in Lemma \ref{l:dot-prod-sign}. Then if 
    \begin{align}\label{e:dot-abs-1}
        |\left<\hat y , \wt \nu \right>| \leq \frac{1}{10} \mbox{ then } |\left< \Omega(y), \nu \right>| \leq \frac{1}{5}.
    \end{align}
      Similarly, if 
      \begin{align}\label{e:dot-abs-2}
      |\left< \hat y, e_L \right>| \leq \frac{1}{10}, \mbox{ then }  |\left< \Omega(y), \Omega(e_L)\right> | \leq \frac{1}{5}.    
      \end{align}
\end{lemma}
\begin{figure}
\centering
\begin{minipage}{.5\textwidth}
  \centering
  \includegraphics[scale=0.5]{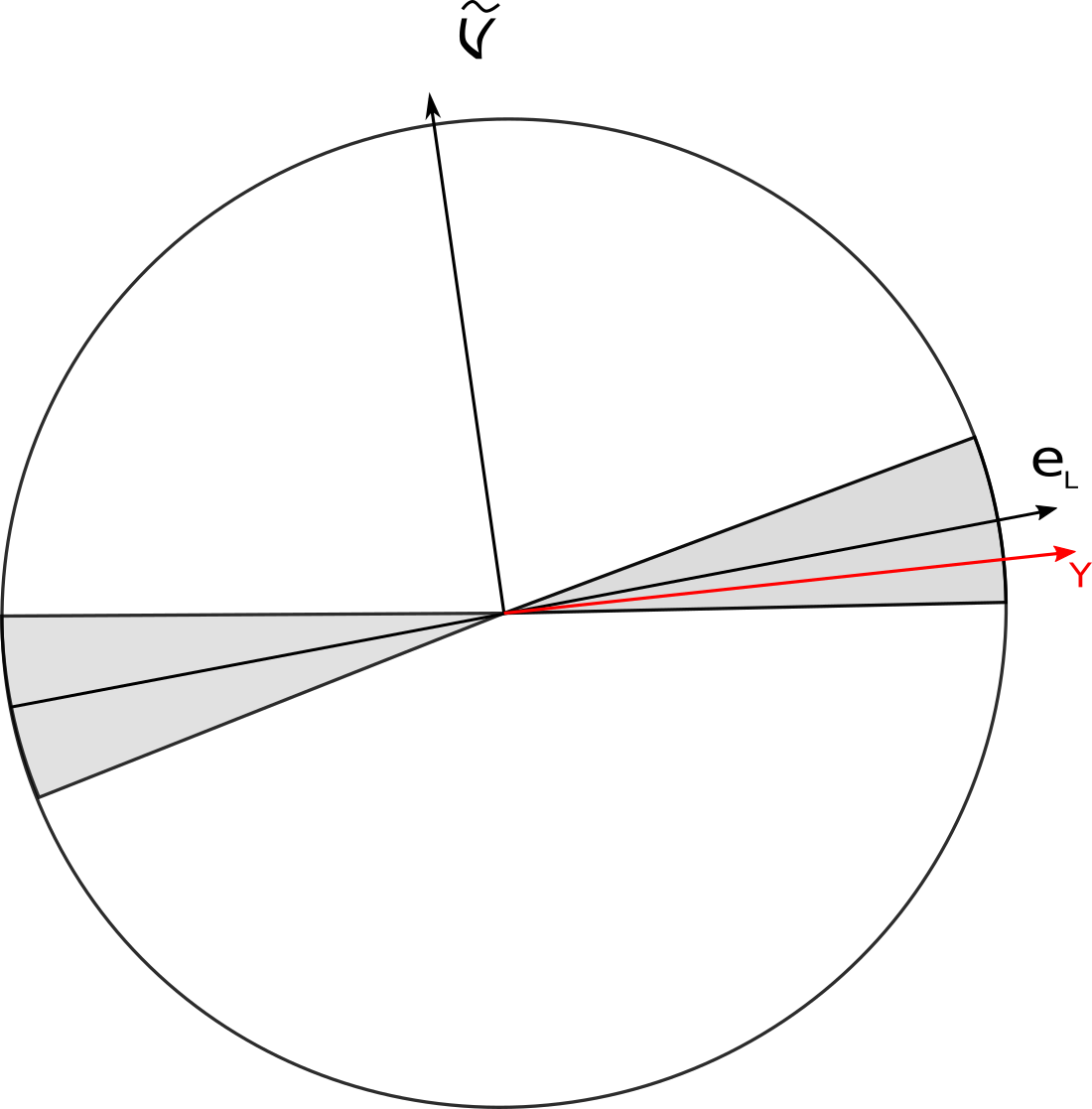}
  \captionof{figure}{When \eqref{e:dot-abs-1} holds, $y$ cannot lie in the white area.}
  \label{f:3}
\end{minipage}%
\begin{minipage}{.5\textwidth}
  \centering
  \includegraphics[scale=0.5]{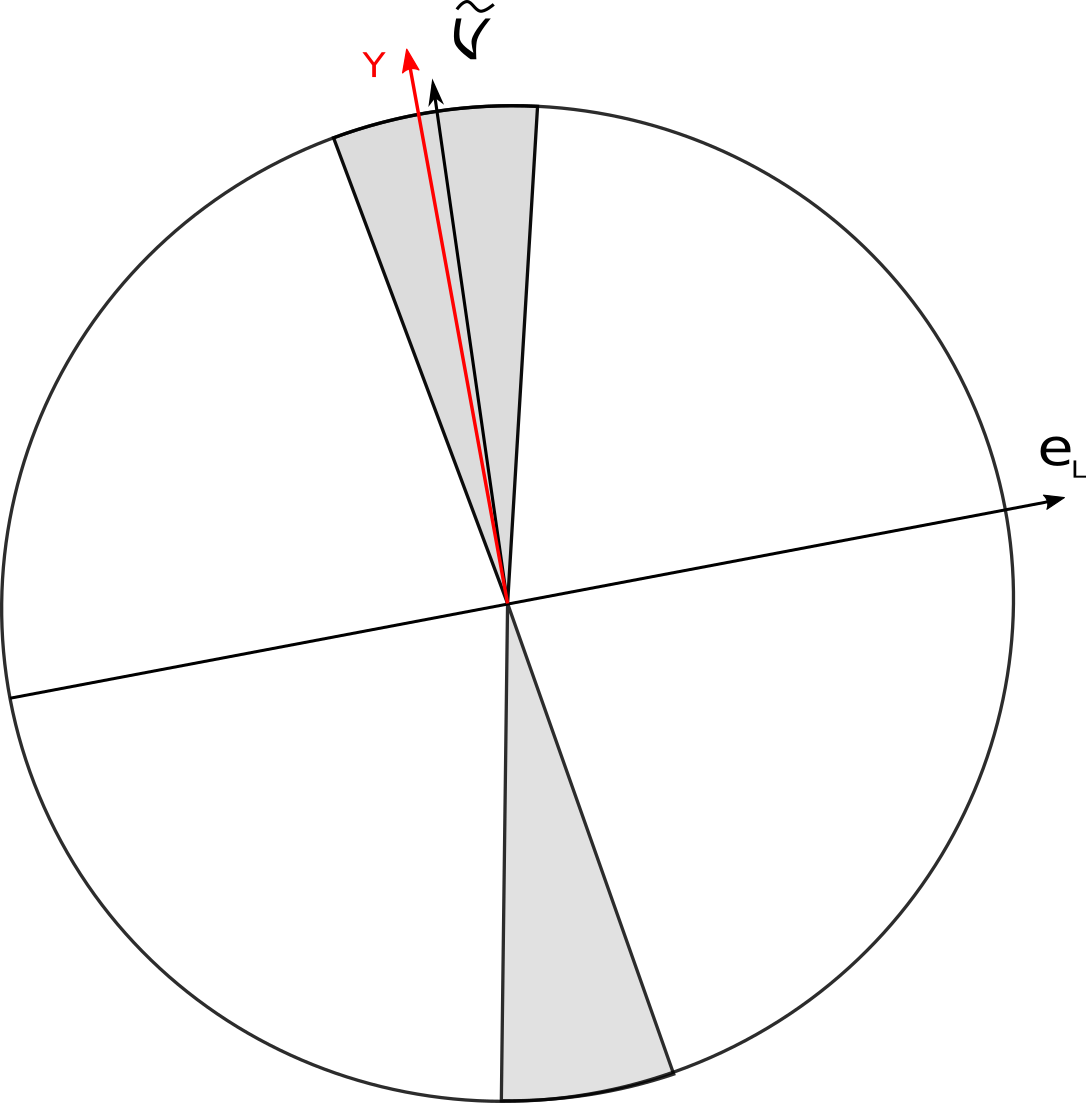}
  \captionof{figure}{This is the case when \eqref{e:dot-abs-2} holds}
  \label{f:4}
\end{minipage}
\end{figure}
\begin{proof}
The proof of this lemma is very similar to the proof of the previous one; we just give a brief sketch. We start with \eqref{e:dot-abs-1}. We can assume that $y \in\C^+$, for otherwise one can argue as in Remark \ref{r:C+}; we can also assume that $y$ is contained in the quadrant $Q_1$. That $|\left< \hat y, \wt \nu\right>| \leq 0.1$ implies that $\hat y$ is contained in an arc of length at most $0.11$ ending at $1$ (see Figure \ref{f:2}). This, the fact that $\Omega$ is bi-Lipschitz and the assumption $\delta_\Omega \leq 1/20$, imply that
\begin{align}\label{e:dot-abs-1eq}
    |\left< \Omega(y), \nu\right>| \leq |\sin((1+\delta_\Omega)\alpha) |\leq |\sin(\alpha)| + \alpha \delta_\Omega \leq 2 |\sin(\alpha)| \leq \frac{1}{5},
\end{align}
where we used Taylor's theorem and the fact that $\frac{\alpha}{20} \leq \sin(\alpha)$ for $0 \leq \alpha \leq \pi/2$. 
Proving \eqref{e:dot-abs-2} can be done similarly: we write
\begin{align*}
    |\left< \Omega(y), \Omega(e_L) \right> | = |\cos(\omega(\theta_y)- \omega(\theta_L))|.
\end{align*}
We have two possibilities. If $|\cos(\omega(\theta_y)-\omega(\theta_L))| \leq \cos((1+\delta_\Omega)|\theta_y - \theta_L|)$, then we compute (with $\alpha=|\theta_y-\theta_L|$),
\begin{align*}
    \cos((1+\delta_\Omega)\alpha) = \cos(\alpha) - \sin(\xi) \alpha \delta_\Omega,
\end{align*}
where $\xi \in [\alpha, (1+\delta_\Omega)\alpha]$. Using the assumption in \eqref{e:dot-abs-2}, that $\delta_\Omega$ and $\alpha \leq \pi/2$, the right hand side is bounded above by $1/10 + \pi/40 \leq \frac{1}{5}$. The case $|\cos(\omega(\theta_y)-\omega(\theta_L))| \leq | \cos( (1+\delta_\Omega)^{-1}\alpha)|$ can be dealt with in the same way. 
\end{proof}
\begin{remark} \label{r:dist-dot}
Note in particular that from \eqref{e:dot-abs-1eq}, we obtain
\begin{align*}
    |\left< K(y), \nu \right>| \leq 2 |\left< y, \wt \nu \right>| \leq \dist(y, L).
\end{align*}
\end{remark}

\newpage
We now use Taylor expansion to split the difference of $C_{\Omega, \phi}(x,r)$ and $C_{\Omega, \phi}(0,r)$ into two terms, one linear and the other which can be controlled in an appropriate manner. 
\begin{lemma} \label{lemma:GK_splitdiff}
Let $0, x\in Q$; let $\phi: \R \to \R_+$ be a $\mathcal{C}^2$, radial function; we have
\begin{align*}
    C_{\Omega, \phi}(x, r) - C_{\Omega, \phi}(0, r) = T(x) + E(x), 
\end{align*}
where $T(x)$ is a term linear in $x$ and $E(x)$ is controlled approriately.
\end{lemma}
\begin{remark}
We have exact expressions for $T$ and $E$, see \eqref{eq:GK_T} and \eqref{eq:SYM_E} below; we will refer to $T$ as `the linear term' and to $E$ as `the error term'. 
\end{remark}
\begin{proof}
We Taylor expand both the smooth cut off and the kernel around 0 as follows. First, we see that
\begin{align*}
    K(x-y) = K(-y) + DK(-y) \cdot x + \frac{1}{2}x^T D^2 K(\xi_{x,0}- y) x, 
\end{align*}
where $\xi_{x, 0}$ is contained in the line segment joining $x$ and $0$ and 
\begin{align*}
    x^T D^2 \Omega (\xi_{x, 0}-y) x = \ps{x^T D^2 \Omega_1(\xi_{x, 0}-y) x , x^T D^2 \Omega_2(\xi_{x, 0}-y) x }
\end{align*}
Now, we let 
\begin{align*}
    s_0:= \left| \frac{K(y)}{r} \right|^2 = \left| \frac{y}{r} \right|^2,
    s := \left|\frac{K(-y) + DK(-y) \cdot x + \frac{1}{2}x^T D^2 K(\xi_{x,0} -y) x}{r} \right|^2. 
\end{align*}
Again by Taylor's expansion, we have that 
\begin{align*}
\phi(s) =\phi (s_0) +\phi'(s_0) (s-s_0) +\phi''(\xi_{s, s_0}) (s_0-s)^2. 
\end{align*}

Applying this to the above difference, we get the following terms:
\begin{align*}
     A:= & \frac{1}{r} \int \frac{K(x-y)}{r}\phi\left(\left|\frac{y}{r}\right|^2\right) \, d \mu(y),
    \\
     B:= & \frac{1}{r} \int \frac{K(x-y)}{r}\phi'\left(\left|\frac{y}{r}\right|^2\right) \, \left[ -\left| \frac{DK(-y) \cdot x + \frac{1}{2} x^T D^2 K(\xi_{x,0} -y) x}{r}\right|^2 \right.  \\
     & \enskip \enskip \left. - 2\left< \frac{K(-y)}{r}, \frac{DK(-y) \cdot x + \frac{1}{2}x^T D^2 K(\xi_{x,0} -y) x}{r} \right> \right] \, d \mu(y),
    \\
    C:= & \frac{1}{r} \int \frac{K(x-y)}{r}\phi''\left(\xi \right) \,\left[ -\left| \frac{DK(-y) \cdot x + \frac{1}{2}x^T D^2 K(\xi_{x,0} -y) x}{r}\right|^2 \right.
    \\ & \enskip \enskip \left.- 2\left< \frac{K(-y)}{r}, \frac{DK(-y) \cdot x + \frac{1}{2}x^T D^2 K(\xi_{x,0} -y) x}{r} \right> \right]^2 \, d \mu(y),
    \\ D:= &  \frac{1}{r} \int \frac{K(-y)}{r}\phi\left(\left|\frac{y}{r}\right|^2\right) \, d \mu(y).
\end{align*}

\subsubsection{Splitting of $A$}
We now expand the kernel to obtain 
\begin{align*}
    A_1 & := \frac{1}{r} \int \frac{K(-y)}{r}\phi\left(\left|\frac{y}{r}\right|^2\right) \, d \mu(y), \\
    A_2 & := \frac{1}{r} \int \frac{DK(-y) \cdot x }{r}\phi\left(\left|\frac{y}{r}\right|^2\right) \, d \mu(y), \\
    A_3 & := \frac{1}{r} \int \frac{\frac{1}{2}x^T D^2(\xi_{x, y}) x}{r}\phi\left(\left|\frac{y}{r}\right|^2\right) \, d \mu(y)
\end{align*}

\subsubsection{Splitting of $B$}
To ease the notation, set 
\begin{align} 
    SO := -\left| \frac{DK(-y) \cdot x + \frac{1}{2}x^T D^2 K(\xi_{x,0} -y) x}{r}\right|^2, \label{e:SO}\\
    DP := - 2\left< \frac{K(-y)}{r}, \frac{DK(-y) \cdot \frac{1}{2}x + x^T D^2 K(\xi_{x,0} -y) x}{r} \right>.\label{e:DP}
\end{align}
Similarly, we expand $K(x-y)$ in $B$ to obtain the terms
\begin{align*}
    B_1 := &  \frac{1}{r} \int \frac{K(-y)}{r}\phi'\left(\left|\frac{y}{r}\right|^2\right) \left[ SO + DP\right] \, d \mu(y), \\
    B_2 := & \frac{1}{r} \int \frac{DK(-y) \cdot x }{r}\phi'\left(\left|\frac{y}{r}\right|^2\right) \left[ SO + DP\right] \, d \mu(y),\\
    B_3 := & \frac{1}{r} \int \frac{x^T D(\xi_{x, y}) x}{r}\phi'\left(\left|\frac{y}{r}\right|^2\right) \left[ SO + DP\right] \, d \mu(y).
\end{align*}
We further split $B_1$ as
\begin{align*}
    B_{1,1}:= & \frac{1}{r} \int \frac{K(-y)}{r}\phi'\left(\left|\frac{y}{r}\right|^2\right) \left[ SO\right] \, d\mu(y), \\
    B_{1,2} := & \frac{1}{r} \int \frac{K(-y)}{r}\phi'\left(\left|\frac{y}{r}\right|^2\right) \left[ DP\right] \, d \mu(y).
\end{align*}

Moreover, we split $B_{1,2}$ as 
\begin{align*}
    B_{1,2,1} := & \frac{1}{r} \int \frac{K(-y)}{r}\phi'\left(\left|\frac{y}{r}\right|^2\right) \left[ \left< \frac{K(-y)}{r}, \frac{DK(-y) \cdot x }{r} \right>\right] \, d \mu(y), \\
    B_{1,2,2}:= &\frac{1}{r} \int \frac{K(-y)}{r}\phi'\left(\left|\frac{y}{r}\right|^2\right) \left[ \left< \frac{K(-y)}{r}, \frac{\frac{1}{2}x^T D^2 K(\xi_{x,0} -y) x}{r} \right>\right] \, d \mu(y).
\end{align*}

Now we set (for fixed $r>0$), 
\begin{align} \label{eq:GK_T}
    T(x) := A_2(x) + B_{1,2,1}(x)
\end{align}
and
\begin{align} \label{eq:SYM_E}
    E(x):= A_3(x) + B_{1,1}(x) + B_{1,2,2}(x) + B_2(x) + B_3(x) + C(x).
\end{align}
\end{proof}

\section{Case when $\beta$'s are small} \label{s:small-beta}
In this section and in the next one we will obtains certain bounds on the $\beta$ numbers (as defined in \eqref{eq:01_beta}, with $p=2$). The method that we use come from a paper of Tolsa, \cite{tolsa2008}. We have a few more terms to deal with because our kernel is not linear. 

We first consider the case where our measure is already quite flat, that is, the $\beta$ numbers are small. In the next section we will look at the case where the $\beta$ numbers are instead large. 

Let $A$ be as in \eqref{eq:SYM:r}, let $N \sim \log(A)$, and let $Q \in \mucubes$ so that $r \sim 2^N \ell(Q) \sim A \ell(Q)$;  let $x_0$ and $x_1$ be the balanced points in $Q$ guaranteed by Lemma \ref{lemma:SYM_balancedcubes}. We will assume throughout this section that 
\begin{align} \label{eq:SYM_betasmall}
    \sum_{k=0}^N \beta_2(B(x_0, 2^k \ell(Q)) \leq \tau, 
\end{align}
where $\tau>0$ is a small constant which will be fixed later.

It will be convenient for us to define the $C_{\Omega, \phi}$ quantity (recall the definition in \eqref{e:C-smooth}) in terms of a smooth cut off of the annulus $A\ps{0, \frac{1}{2}, 2}$: let $\chi_\frac{1}{2}$ be a smooth radial function supported on $B(0,1)$ and so that $$\chara_{B\ps{0, \frac{1}{2}}} \leq \chi_{\frac{1}{2}} \leq \chara_{B(0,1)};$$ let $\chi_1$ be the same smooth radial function rescaled so that $\chara_{B(0,1)}\leq \chi_1\leq \chara_{B(0,2)}$, that is, we put $\chi_1(x):= \chi_{\frac{1}{2}}\ps{\frac{x}{2}}$.  We set
\begin{align} \label{eq:SYM_smoothcut}
    \phi:= \chi_1 - \chi_{\frac{1}{2}}.
\end{align}
Then we see that, for $r>0$, $\phi\ps{\av{ \frac{x-\cdot}{r}}^2}$ is supported on $A\ps{x, \frac{r}{2}, 2r}$. 

Recall Lemma \ref{lemma:SYM_borelequiv} and Remark \ref{remark:SYM_borelequiv}: $C_{\Omega, \phi}(x,r) = 0$ if and only if $C_{\Omega}(x,r)=0$; throughout this section we will work with the quantity $C_{\Omega, \phi}$ with $\phi$ as in \eqref{eq:SYM_smoothcut}.

\subsection{Lower bounds on the linear component $\mathbf{T}$} \

We denote by $\Pi_L$ the standard orthogonal projection onto the line $L$.
\begin{lemma} \label{lemma:SYMg_lowboundT}
Fix a cube $Q$ and let $x_0$ and $x_1$ be two balanced points of $Q$ and let $L_Q$ be the corresponding balanced line \textup{(}see \eqref{e:balanced-line} for definitions\textup{)}. Let $T$ be defined as in \eqref{eq:GK_T}, with $\phi$ as in \eqref{eq:SYM_smoothcut}. For any $z \in 3Q \setminus L_Q$, set 
\begin{align} \label{e:e_z}
    e_z= \frac{z- \Pi_{L_Q}(z)}{|z-\Pi_{L_Q}(z)|}.
\end{align}
If \eqref{eq:SYM_betasmall} holds and $\delta_\Omega$ is sufficiently small, we have 
\begin{align}
    - \left< T(e_z), \Omega(e_z) \right> \gtrsim \frac{1}{r}. 
\end{align}
\end{lemma}

We will prove this lemma through some sublemmata. But first, note that without loss of generality, we can work with 
\begin{align} \label{e:x_0}
    x_0 = 0;
\end{align}
note also that $e_z$ is perpendicular to $L_Q$.

Let $z \in 3Q$ and set $\bar z := \Pi_{L_Q}(z)$.
Let us look at the quantity
\begin{align*}
    \langle T(e_z), \nu \rangle & = \langle A_2(e_z), \nu \rangle + \langle B_{1,2,1}(e_z), \nu \rangle. 
\end{align*}
(Recall the definition of $T$ in \eqref{eq:GK_T}). We will first take care of the term containing $A_2$ and later the one containing $B_{1,2,1}$.

\subsubsection{Lower bounds on $A_2$}
Recall that 
\begin{align*}
    A_2(x)= \frac{1}{r} \int \frac{DK(-y) \cdot x}{r} \phi \ps{\av{\frac{y}{r}}^2} \, d\mu(y).
\end{align*}
Our next short term goal is to find a lower bound for the first term of $T$ (as in \eqref{eq:GK_T}). 
\begin{sublemma} \label{sublemma:SYMg_lowA2}
Recall that $x_0=0$ and  $x_1$ are two balanced points in $Q$ and $L_Q$ is the corresponding balanced line \textup{(}see \eqref{e:balanced-line}\textup{)}. Let $\nu$ denote  the vector $\Omega(x_1)$ rotated by 90 degrees counter-clockwise \textup{(}we are just choosing the orientation of the normal to $K(L_Q)$, so that it agrees with \eqref{e:y-perp}\textup{)}. Then
\begin{align}
    - \left< A_2(e_z), \nu\right> \gtrsim \frac{1}{r}.
\end{align}
    \textup{(}Recall that $e_z$ has been defined in \eqref{e:e_z}\textup{)}.
\end{sublemma}
\begin{proof}

Using \eqref{e:DK-expression}, we write
\begin{align}
    & -\left< A_2(e_z), \nu \right> \nonumber\\
    & = \int \left<\hat y, e_z\right>\left< \Omega(y), \nu \right> \frac{\phi \ps{\av{\frac{y}{r}}^2}}{r^{2}}\, d\mu(y) + \int \left< e_z, \hat y_\perp\right> \left<d_{\hat y}\Omega \cdot \hat y_\perp , \nu \right>  \frac{\phi \ps{\av{\frac{y}{r}}^2}}{r^{2}}\, d\mu(y) \nonumber\\
    & =: I + II.\label{eq:SYMg_A2low}
\end{align}
For the sake of clarity, we consider a few cases separately. Denote $\C^+$ the half space `above' $L$, and $\C^-$ the half space below. Denote $Q_1$ the first quadrant in the plane with basis $e_{x_1}, e_z$, and $Q_2$ the second quadrant.
\begin{itemize}
    \item \underline{Suppose that $y, z \in \C^+$}. Under these assumptions,
we look first at the integrand of $I$: suppose that
\begin{align} \label{eq:SYMg_1}
    |y| \geq \left< y, e_z\right> \geq \frac{|y|}{10}. 
\end{align}
From Lemma \ref{l:dot-prod-sign} (applied with $e_z =\wt \nu$), in particular \eqref{e:dot1}-\eqref{e:dot1b},  we obtain
\begin{align}\label{e:A2-bound1}
    \left< |y|\Omega(y), \nu \right> \geq \frac{|y|}{20}.
\end{align}
 
We consider now the integrand of $II$, that is, the term $\left< e_z,  \hat y_\perp\right> \left< d_{\hat y}\Omega \cdot \hat y_\perp, \nu\right>$. In analogy to \eqref{eq:SYMg_1}, suppose that
\begin{align} \label{e:dot-y-perp}
|y| \geq |\left< |y| \hat y_\perp, e_z \right>| \geq \frac{|y|}{20}.
\end{align}
\begin{itemize}
    \item 
If $y \in Q_1$, then 
\begin{align*}
    \eqref{e:dot-y-perp} \implies |y| \geq \left< |y| \hat y_\perp, e_z \right> \geq |y|/20.
\end{align*}
Note that the mapping sending $\hat y \mapsto \hat y_\perp$ is an isometry, hence 
\begin{align*}
    |y|\left< \hat y_\perp, e_z \right> = \left< y, e_{x_1}\right>.
\end{align*}
Now, let us denote by $\Omega(y)_\perp$ the vector $\Omega(y)$ rotated by 90 degrees counterclockwise (so that $\Omega(y)_\perp \perp \Omega(y)$ and $|\Omega(y)_\perp|= |\Omega(y)|=1$); we see that both $d_{\hat y} \Omega \cdot \hat y_\perp$ and $\Omega(y)_\perp$ lie in $T_{\Omega(\hat y)} \bS$ and moreover, since $\hat y_\perp \in T_{\hat y} \bS$,
\begin{align*}
   \frac{1}{1+\delta_\Omega} \leq |d_{\hat y} \Omega \cdot \hat y_\perp| \leq 1+ \delta_\Omega,
\end{align*} since $\Omega$ is bi-Lipschitz. In particular, for a point $u \in \C$, if we let $\theta$ be the angle between $u$ and $d_{\hat y}\Omega \cdot \hat y_\perp$ and $\theta'$ be the angle between $u$ and $\Omega(y)_\perp$, first, 
\begin{align*}
    \theta =\theta',
\end{align*}
and second, 
\begin{align}
    \left< d_{\hat y}\Omega \cdot \hat y_\perp, u \right>&  = |d_{\hat y}\Omega \cdot \hat y_\perp | |u|\cos(\theta) 
     = |d_{\hat y}\Omega \cdot \hat y_\perp | |u|  \cos(\theta') \nonumber \\
    & \sim_{\delta_\Omega} |\Omega(y)_\perp||u| \cos(\theta') 
    = \left< \Omega(y)_\perp, u \right>. \label{e:dot-prods-diff-nodiff} 
\end{align}
Thus we write
\begin{align} \label{e:A2-1}
    \left< y, e_{x_1}\right> \left< d_{\hat y} \Omega \cdot \hat y_\perp, \nu \right> \sim_{\delta_\Omega} \left< y, e_{x_1}\right> \left< \Omega(y)_\perp, \nu \right> = \left< y, e_{x_1}\right> \left< \Omega(y), \Omega(x_1) \right>,
\end{align}
as $\nu \perp \Omega(x_1)$ by definition. At this point, we apply Lemma \ref{l:dot-prod-sign} (here $e_{x_1}=e_L)$, in particular \eqref{e:dot2}-\eqref{e:dot2b}, so to obtain 
\begin{align} \label{e:A2-bound2}
   1 \geq  \left<\Omega(y), \Omega(x_1) \right> \geq \frac{1}{20}.
\end{align}
\item If $y \in Q_2$, then \begin{align*}
     \eqref{e:dot-y-perp} \implies - |y| \leq \left< |y|\hat y_\perp, e_z \right> \leq -|y|/10.
\end{align*}
With the same arguments as above (see also Remark \ref{r:C+}, we obtain that 
\begin{align} \label{e:dot11}
    -1 \leq \left< \Omega(y), \Omega(x_1)\right> \leq - \frac{1}{20}. 
\end{align}
\end{itemize}
Thus, whenever $y,z \in \C^+$, and we have that $y,z$ are so that \eqref{eq:SYMg_1} and \eqref{e:dot-y-perp} are satisfied, then the integrand in both $I$ and $II$ are positive and in fact bounded away from 0. 
\item 
\underline{Suppose that $y \in \C^+$ and $z \in \C^-$.} In this case,  it suffices to go through the arguments above by taking dot products with $-e_z$ and $-\nu$. 
Note that if $|y| \geq \left< y, -e_z \right> \geq|y|/10$ (resp. $|y| \geq \left< \hat y_\perp, -e_z \right> \geq |y|/10$), then 
\begin{align} \label{eq:SYMg_symmetry}
& \left< \Omega(y), \Omega(-e_{x_1}) \right>= - \left< \Omega(y), \Omega(e_{x_1}) \right> \geq \frac{1}{20} \\
& \mbox{resp.  } - \left< \Omega(\hat y_\perp), \nu \right> \geq \frac{1}{20}.
\end{align}
\item The remaining cases can be dealt with by symmetry.
\end{itemize}
Set 
\begin{align}
    & G := \{y \in \C \, |\, |\left< \hat y, e_z \right> | \geq \frac{1}{10} \mbox{ and } \, |\left< \hat y_\perp, e_z\right>| \geq \frac{1}{10} \};\nonumber \\
    & H:= \{y \in \C \,|\, |\left<\hat y, e_z \right>| \leq \frac{1}{10} \};\nonumber \\
    & F := \{ y \in \C \, |\, |\left< \hat y_\perp, e_z \right>| \leq \frac{1}{10}\}.\label{e:GHF}
\end{align}
Clearly these sets are disjoint, and their union is $\C$.
Thus, we may write
\begin{align*}
    \eqref{eq:SYMg_A2low} = \int_G + \int_H + \int_{F},
\end{align*}
where we are still integrating with respect to $\phi(|y|^2/r^2)\, \frac{d\mu(y)}{r^2}$. 
The first integral can be bounded below using the previous discussion (namely \eqref{e:A2-bound1}, \eqref{e:A2-bound2}, \eqref{e:dot11}, \eqref{eq:SYMg_symmetry}), and using symmetry: we obtain
\begin{align*}
\int_G \gtrsim \frac{1}{200 r^2} \mu(G \cap(A(0, r/2, 2r))).    
\end{align*}
As for the other integrals, we see that  
\begin{align*}
    \mbox{if } |\left<\hat y_\perp, e_z\right>| \leq 0.1 \mbox{ then } |\left< \Omega(\hat y)_\perp, \nu\right>| \leq 0.2;
\end{align*}
this follows using Lemma \ref{l:dot-abs}, in particular \eqref{e:dot-abs-1}; similarly, if $|\left< \hat y, e_z\right>|\leq 0.1 $, then  $|\left<\Omega(y), \nu\right>|\leq 0.2$. 
But note that if $|\left< \hat y, e_z \right>| \leq 0.1$, then 
\begin{align*}
    |\left< \hat y_\perp, e_z \right>| = |\left< \hat y, e_{x_1}\right>| \geq (1-0.01)^{\frac{1}{2}}. 
\end{align*}
Similarly, if $|\left< \Omega(\hat y), \nu \right>| \leq 0.2$, then 
\begin{align*}
    |\left< \Omega(y)_\perp, \nu \right>| \gtrsim (1-0.02)^{\frac{1}{2}}.
\end{align*}
This implies that, for $y \in F$ we have that
\begin{align} \label{e:A2-2}
    & \left<\hat y,e_z\right>\left< \Omega(y), \nu \right>  +  \left< \hat y_\perp, e_z \right> \left< d_{\hat y}\Omega \cdot \hat y_\perp, \nu\right> 
     \gtrsim (1-0.01) - 0.02 \geq \frac{1}{2},
\end{align}
where we also used \eqref{e:dot-prods-diff-nodiff}.
Thus, 
\begin{align*}
    \int_{F} \gtrsim \frac{1}{r^2} \mu(H \cap(A(0, r/2, 2r))), 
\end{align*}
and similarly for $\int_H$. Now using the lower regularity of $\mu$, we obtain the lemma. 
\end{proof}
\subsubsection{Control on $B_{1,2,1}$}
Recall that $B_{1,2,1}(x)$ is given by 
\begin{align*}
    \frac{1}{r} \int \frac{K(-y)}{r}\phi'\left(\left|\frac{y}{r}\right|^2\right) \left[ \left< \frac{K(-y)}{r}, \frac{DK(-y) \cdot x }{r} \right>\right] \, d \mu(y)
\end{align*}
\begin{sublemma}\label{lemma:SYM_bound_on_B121}
Keep the notation as above. In particular, recall that $x_0=0$ and $x_1$ are balanced points of $Q$ (and so $0,x_1 \in L_Q$ - which is the balanced line), that $z \in 3Q$, $e_z$, as given in \eqref{e:e_z} is perpendicular to $L_Q$ and that $\nu$ is the normal vector to $K(L_Q)$. Then 
\begin{align} \label{eq:SYM_bound_on_B121}
    - \int_{B(0,r)} \dist(y, L_Q)^2 r^{-4}\, d\mu(y) \leq \left<B_{1,2,1}(e_z), \nu \right> \leq \int_{B(0,r)} \dist(y, L_Q)^2 r^{-4} \, d\mu(y).
\end{align}
\end{sublemma}
\begin{proof}
As before, we split the integral into the radial derivative and the spherical one and we apply Lemma \ref{sl:DK-expression}:
\begin{align}
    B_{1,2,1}(e_z) & = \frac{1}{r} \int \frac{K(y)}{r}\phi'\left(\left|\frac{y}{r}\right|^2\right) \left[ \left< \frac{K(y)}{r}, \frac{\left< \hat y, e_z \right>\Omega(y)}{r} \right>\right] \, d \mu(y) \\
    & + \frac{1}{r} \int \frac{K(y)}{r}\phi'\left(\left|\frac{y}{r}\right|^2\right) \left[ \left< \frac{K(y)}{r}, \frac{\left<e_z, \hat y_\perp\right>  d_{\hat y} \Omega\cdot \hat y_\perp}{r} \right>\right] \, d \mu(y)\\
    & =: I(x) + II(x).
\end{align}
Note that since $d_{\hat y} \Omega \cdot \hat y_\perp \in T_{\Omega(y)} \bS$ and $\Omega(y) \in \ps{T_{\Omega(y)} \bS}^\perp$, we have 
\begin{align} \label{eq:SYMg_IIzero}
    \left< II(e_z), \nu \right> = 0.
\end{align}
Let us now compute $\langle I(e_z), \nu \rangle$. We have
\begin{align}\label{eq:SYMg_Idist}
\langle I(e_z), \nu \rangle & = \frac{1}{r} \int \frac{\langle K(y), \nu \rangle}{r} \phi' \left(\left|\frac{y}{r}\right|^2\right) 
\left[ \left< \frac{K(y)}{r}, \frac{\left< \hat y, e_z \right> \Omega(y)}{r} \right>\right] \, d \mu(y).
\end{align}
The second inner product in the integral can be re-written as 
\begin{align*}
    \left< \hat y, e_z\right> \left< r^{-1}K(y), r^{-1} \Omega(y)\right> =  \frac{\left< \hat y, e_z \right>|y|}{r^2} = \frac{\left< y, e_z\right> }{r^2},
\end{align*}
and so  we obtain 
\begin{align*}
    \left< I(e_z), \nu \right> = \frac{1}{r^{4}} \int \left<  y, e_z\right> \left< K(y), \nu \right> \, \phi'\ps{ \av{\frac{y}{r}}^2} \, d\mu(y).
\end{align*}
It immediately follows from Remark \ref{r:dist-dot} that
\begin{align*}
    - \int_{B(0,r)} \dist(y, L_Q)^2 r^{-4}\, d\mu(y) \leq \eqref{eq:SYMg_Idist} \leq \int_{B(0,r)} \dist(y, L_Q)^2 r^{-4} \, d\mu(y).
\end{align*}
This proves the Sublemma. 
\end{proof}

\begin{corollary} \label{corollary:SYM_lowerbound}
\begin{align*}
    & \langle A_2(e_z) + B_{1,2,1}(e_z) , \nu \rangle \\
    & \gtrsim   \frac{1}{r}-\frac{2}{r^2} \int_{B(0,r)} \left( \frac{\dist(y, L_{Q})}{r}\right)^2  \, d \mu(y).
\end{align*}
\end{corollary}
\begin{proof}
This follows at once from Sublemma \ref{sublemma:SYMg_lowA2} and Sublemma \ref{lemma:SYM_bound_on_B121}.
\end{proof}

We divert for a moment from the main argument, to delve a little more in the choice of the balanced cubes $x_0$ and $x_1$. The following definition will be used later on. 
For a point $y \in \spt(\mu)$, set
\begin{align}
    \beta(y, Q) := \ps{\int_{A \ell(Q)}^{2A \ell(Q)} \beta_{2,\mu}(y,r)^2 \, \dr}^{\frac{1}{2}}.
\end{align}

\begin{lemma} \label{lemma:GK_goodpoints}
There exists a constant $c^*>0$ such that the following holds. Let $Q \in \mucubes$ and let $P(Q)$ be the line minimising $\beta_{\mu,2}(Q)$;  we can choose $x_0, x_1$ as in Lemma \ref{lemma:SYM_balancedcubes} so that first,
\begin{align}
    \dist(x_j, P(Q)) \leq c^* \beta_{\mu, 2}(Q) \ell(Q), \, \, j=0,1,
\end{align}
and second,
\begin{align}
    \beta(x_j, Q)^2 \leq c^* \fint_Q \beta(y, Q)^2 \, d\mu(y), \,\, j=0,1.
\end{align}
\end{lemma}
\begin{proof}
Let $y_0,y_1 \in Q$ be the balanced points guaranteed by Lemma \ref{lemma:SYM_balancedcubes} and, for a constant $0< \eta<1$ to be fixed later, let $B_j$ denote $B(y_j, \eta \ell(Q)) \cap Q$. For two constants $c, c'>0$, we set
\begin{align*}
    & F_j:= \left\{ y \in B_j |\dist(y, P(Q)) \leq c\, \beta_{\mu, 2}(Q)\right\},\\
    & G_j := \left\{ y \in B_j | \, \beta(y, Q)^2 \leq c'\,  \fint_Q \beta(y, Q)^2 \, d\mu(y) \right\}.
\end{align*}
By Chebyshev inequality and Cauchy - Schwarz, we have, for $j=0,1$, 
\begin{align} \label{eq:SYM_cheby_Fj}
    \mu(B_j \setminus F_j) & = \mu\ps{\left\{y \in B_j\, |\, \ell(Q)^{-1} \dist(y,P(Q)) > c \, \beta_{\mu,2}(Q) \right\}} \nonumber \\
    & \leq \frac{1}{c \, \beta_{\mu,2}(Q)} \, \int_{\{\ell(Q)^{-1} \dist(y, P(Q)) > c\, \beta_{\mu,2}(Q)\}} \frac{\dist(y, P(Q))}{\ell(Q)} \, d\mu(y) \nonumber\\
    & \leq  \frac{1}{c \, \beta_{\mu,2}(Q)} \mu(B_Q)^{\frac{1}{2}} \ps{\int_{B_Q} \ps{\frac{\dist(y, P(Q))}{\ell(Q)}}^2 \, d\mu(y) }^{\frac{1}{2}} \nonumber \\
    & \lesssim \frac{\mu(B_Q)}{c}.
\end{align}
Note that by Ahlfors regularity, we have that $\mu(B_Q)\lesssim \frac{1}{\eta} \mu(B_j)$.
Similarly, we have
\begin{align} \label{eq:SYM_cheby_Gj}
    \mu(B_j \setminus G_j) & \leq \frac{1}{c' \fint_Q \beta(y, Q)^2 d\mu} \int_Q \beta(y, Q)^2 \, d\mu = \frac{\mu(Q)}{c'} \lesssim \frac{\mu(B_j)}{c' \eta}.
\end{align}
We want to show that, for $c, c'$ large enough, we have that $F_j \cap G_j \neq \emptyset$; we argue by contradiction and we assume that for all $c, c'$ the intersection is empty. By \eqref{eq:SYM_cheby_Fj} and \eqref{eq:SYM_cheby_Gj}, we have
\begin{align}
    \frac{\eta r}{c_0} \leq \mu(B_j)  \leq \mu(B_j \setminus F_j) + \mu(B_j \setminus G_j) \leq \mu(B_j)\ps{\frac{1}{c \eta} + \frac{1}{c' \eta} } \leq c_0 r \ps{\frac{1}{c}+\frac{1}{c'}}.
\end{align}
This is clearly a contradiction. Thus for $c,c'$ large enough (depending only on $\eta$ and $c_0$), we see that $F_j \cap G_j \neq \emptyset$. Taking $c^*= c^*(\eta, c_0) := \max\{c,c'\}$ we end the proof of the lemma. 
\end{proof}

\begin{proof}[Proof of Lemma \ref{lemma:SYMg_lowboundT}]
Recall the notation for balanced lines as in \eqref{e:balanced-line} and the lines below it. We see that
\begin{align} \label{eq:GK_distl0}
\dist(y, L_Q) \leq \dist(y, L_{B(0, r)}) + \dist_H\left( L_{B(0, r)} \cap B(0, r), L_Q \cap B(0, r)\right), 
\end{align}
and, using Lemma \ref{lemma:GK_goodpoints}, that
\begin{align} \label{eq:GK_distl0b}
\dist_H\left( L_{B(0, r)} \cap B(0, r), L_Q \cap B(0, r)\right) \lesssim r \, \sum_{k=0}^N \beta_{\mu, 2} (0, 2^k \ell(Q)),
\end{align}
with $N \sim \log_2(A)$.
Recalling that $\mu$ is Ahlfors 1-regular, we have that
\begin{align*}
& r^{-1} \, \int_{B(0, r)} \left( \frac{\dist(y, L_Q)}{r}\right)^2 \, d\mu(y) \\
& \lesssim r^{-1} \int_{B(0, t)} \left( \frac{\dist(y, L_{B(0, r)})}{r}\right)^2 \, d\mu(y)\\
& \enskip \enskip + r^{-1} \, \int_{B(0, r)} \left( \sum_{k=0}^N \beta_{\mu, 2}(0, 2^k \ell(Q))\right)^2 \, d\mu(y) \\
& \lesssim \beta_{\mu, 2}(0, 2^N \ell(Q))^2 + \frac{\mu(B(0, r))}{r} \left(\sum_{k=0}^N \beta_{\mu, 2}(0, 2^k \ell(Q))\right)^2 \\
& \lesssim \left(\sum_{k=0}^N \beta_{\mu, 2} (0, 2^k \ell(Q))\right)^2.
\end{align*}
Hence 
\begin{align*}
    \frac{1}{ r^{2}} \int_{B(0, r)} \frac{\dist(y,L_Q)^2}{r^2} \, d\mu(y)  \lesssim \frac{1}{ r} \, \left(\sum_{k=0}^N \beta_{\mu,2}(0, 2^k\ell(Q))\right)^2 \lesssim \frac{\tau^2}{r}.
\end{align*}
Using Corollary \ref{corollary:SYM_lowerbound}, we finally obtain
\begin{align} \label{eq:SYM_Tlowerbound}
    - \langle T(e_z) , \nu\rangle \geq C \left( \frac{c_0}{r} - \frac{\tau^2}{r} \right) \geq \frac{C(\tau)}{2r}.
\end{align}
An appropriate choice of $\tau$ gives the lemma.
\end{proof}

\subsection{Upper bound on the (nonlinear) term $\mathbf{E}$}
Keep the notation as above. In this subsection we will prove an upper bound for the error term $E$ in terms of distance to the balanced plane $L_Q$.

\begin{lemma} \label{l:error-x}
Let $x_0, x_1$ be the balanced points given by Lemma \ref{lemma:SYM_balancedcubes} and chosen as in Lemma \ref{lemma:GK_goodpoints}. We have that
\begin{align}
    |\left< E(x_1), \nu\right>|  \lesssim \frac{|x_1|^2}{r^{4}}\int_{B(x_0, r)} \dist(y,L_Q) \, d\mu(y).
\end{align}
\end{lemma}
We will prove this lemma in the following several paragraphs; each paragraph corresponds to a piece of $E$. 
Let us recall that
\begin{align*}
    E(x)= A_3(x) + B_{1,1}(x) + B_{1,2,2}(x) + B_2(x) + B_3(x) + C(x), 
\end{align*}
where the various terms were defined in Section \ref{s:split}. 
\begin{remark}
Without loss of generality, we may let $x_0=0$; also, to ease the notation we let $$x_1= x.$$
\end{remark}

\subsubsection{Estimates on $A_3$} \label{subsection:A3}
Recall that 
\begin{align*}
    A_3(x)= \frac{r^{-2}}{2} \int x^T D^2(K(\xi_{x, 0}-y)) x \, \phi(|y|^2/r^2) \, d\mu(y).
\end{align*}
To estimate this term, we want to understand the integrand
\begin{align*}
    x^T D^2 K(\xi_{x,0}-y) x = \left(x^T D^2 K_1(\xi_{x,0}-y) x,  x^T D^2K_2(\xi_{x,0}-y) x\right)^T.
\end{align*}
For each $i=1,2$, and writing $\xi:=  \xi_{x,0}-y$, we have
 \begin{align*}
       & x^TD^2 K_i(\xi)x 
        \\
        &= x^T(D^2|\xi|) \Omega_i (\xi)x +  x^T \left(\left(D|\xi|\right)^T\, D \Omega_i (\xi)+ D(|\xi|) \left(D \Omega_i(\xi)\right)^T \right) x  + x^T |\xi| D^2\Omega_i(\xi)x\\
        & =: I_i(\xi, x) + II_i(\xi, x) + III_i(\xi,x).
    \end{align*}
We write 
\begin{align} \label{e:A3-3}
    I(\xi, x)= 
    \begin{pmatrix}
    I_1(\xi,x) \\ I_2(\xi, x)
    \end{pmatrix} 
    \in \C;
\end{align} 
we write $II$ and $III$ in the same manner.   
\subsubsection{Bounds for $I$}    
    Let us look at $\av{\left< I(\xi,x), \nu\right>}$. A computation tells us that
    \begin{align*}
       \left< I(\xi,x), \nu \right> = \frac{1}{|\xi|} \left< K(\xi), \nu \right> x^T D^2(|\xi|) x.
    \end{align*}
    Now, on one hand, using Remark \ref{r:dist-dot} we see that
    \begin{align}\label{e:A3-aI}
        \frac{\left| \left< K(\xi), \nu \right> \right|}{|\xi|} \lesssim \frac{\av{ \left< \xi, e_z \right> }}{|\xi|} \leq  \frac{\dist(\xi, L_{Q})}{|\xi|}.
    \end{align}
    On the other hand, because $|D^2(|\xi|)| \lesssim |\xi|^{-1}$,
    $
        |x^T D^2(|\xi|) x| \lesssim |x|^2 \frac{1}{|\xi|}.
    $
    Hence, 
    $
        |\langle I(\xi,x), \nu \rangle| \lesssim \frac{\dist(\xi, L_Q) |x|^2}{|\xi|^2},
    $
    and therefore, 
    \begin{align} \label{eq:SYMg_estI}
         \frac{1}{r} \int \left| \frac{ \left< I(\xi, x), \nu \right>}{r} \right| \phi\ps{\av{\frac{y}{r}}^2} \, d\mu(y)
        \lesssim \frac{|x|^2}{r^{4}} \int_{B(0,r)} \dist(y, L_{Q}) \, d\mu(y).
    \end{align}
Note that here we have used that 
\begin{align}\label{e:xi-y-norm}
    |\xi|=|\xi_{0, x}-y| \geq \frac{1}{2} |y| \sim r,
\end{align}
by the definition of the smooth cut off $\phi$ (see \eqref{eq:SYM_smoothcut}.
\subsubsection{Bounds for $II$}\label{sss:A3-II} We now look at $\langle II(\xi, x) , \nu\rangle$. 
    To do so, we consider first the term $\left(D|\xi|\right)^T D \Omega_i (\xi)$. The other one can be dealt with in the same fashion. 
    First, notice that, for $i =1,2$, 
    \begin{align*}
        x^T (D|\xi|)^T D\Omega_i(\xi) x = \langle x, D|\xi| \rangle \langle D \Omega_i(\xi) , x \rangle.
    \end{align*}
    Thus we have that
    \begin{align} \label{eq:SYM_ErrorB_1}
        \left< \left(x^T (D|\xi|)^T D \Omega_1(\xi) x, \, x^T (D|\xi|)^T D \Omega_2(\xi) x\right), \nu \right> = \left< x, D|\xi| \right> 
        \left< 
        \begin{pmatrix}
        \langle D\Omega_1(\xi), x \rangle \\
        \langle D\Omega_2(\xi), x \rangle 
        \end{pmatrix},
        \nu \right>.
    \end{align}
    Let us look at the second inner product above: we may write $$\left(\langle D\Omega_1(\xi), x \rangle ,\langle D\Omega_2(\xi), x \rangle \right)^T$$ as $
        D \Omega(\xi) \cdot x$.
    Then recall Remark \ref{remark:BT_remdiff} and in particular the expression for $D \Omega(\xi)$ in \eqref{e:der-sphere}: we see that \begin{align*}
        D\Omega(\xi) \cdot x = \frac{1}{|\xi|} \left< x, \hat \xi_\perp\right> d_{\hat \xi}\Omega \cdot\hat \xi_\perp;
    \end{align*}
    note that 
    \begin{align*}
        \av{d_{\hat \xi} \Omega \cdot \hat \xi_\perp} \leq  \av{d_{\hat \xi} \Omega} \lesssim 1,
    \end{align*}
    since $|\hat \xi|=1$ by definition, and $\Omega$ has Lipschitz constant close to one. 
    Then, denoting by $x_\perp$ the vector $x$ rotated by $90$ degrees clockwise (so that it is parallel to $e_z$, as defined in \eqref{e:e_z}), we obtain
    \begin{align} \label{eq:SYM_bound_on_der}
       & |\left< D\Omega(\xi) \cdot x, \nu \right> | = \frac{1}{|\xi|}\av{ \left< x, \hat \xi_\perp \right>\left<d_{\hat \xi}\Omega \cdot \hat \xi_\perp , \nu \right> } \lesssim \frac{1}{|\xi|}\av{\ip{x, \hat \xi_\perp}}\nonumber \\
        & = \frac{1}{|\xi|^2} \av{ \left< x^\perp, \xi\right> }= \frac{|x|}{|\xi|^2}\left<e_z, \xi \right> = \frac{|x|}{|\xi|^2}\dist(\xi, L_Q).
    \end{align}
    On the other hand, as far as the first inner product in the right hand side of \eqref{eq:SYM_ErrorB_1} is concerned, we immediately see that it's bounded above by $|x|$. Thus, using again \eqref{e:xi-y-norm}, 
    \begin{align}
        \left| \eqref{eq:SYM_ErrorB_1} \right| \leq \frac{|x|^2}{|y|^2} \dist(\xi, L_Q). 
    \end{align}
    Since $\xi = \xi_{0,x}- y$, and $\xi_{0,x} \in [0, x] \subset L_Q$, we also have that $\dist(\xi, L_Q) = \dist(y, L_Q)$. We therefore conclude (together with the remark that one can deal with the second term in $II(\xi,x)$ in exactly the same way) that 
    \begin{align*}
        \frac{1}{r} \int \av{ \frac{\left<  II(\xi, x), \nu\right> }{r} }\, \phi\ps{\av{\frac{y}{r}}^2} \, d \mu(y)& \lesssim \frac{1}{r} \int \frac{ |\left<x, D(|\xi|)\right>| \, |\left< D \Omega(\xi) x, \nu \right> | }{r} \phi\ps{\av{\frac{y}{r}}^2} \, d \mu(y) \\
        & \lesssim \frac{|x|^2}{r^4} \int_{B(0,r)} \dist(y, L_Q) \, d \mu(y).
    \end{align*}
 
\subsubsection{Bounds for $III$} \label{sss:A3-III}
Let us now look at 
\begin{align} \label{eq:GK_A3_3}
    & \av{\frac{1}{r} \int \frac{\ip{III(\xi, x), \nu}}{r} \, \phi \ps{\av{\frac{y}{r}}^2} \, d\mu(y)} \nonumber \\
    & = \av{ \frac{1}{r} \int \frac{ \ip{\ps{x^T |\xi| D^2\Re \Omega(\xi)x, \, x^T |\xi| D^2 \Im \Omega(\xi) x}^T, \, \nu }}{r} \phi \ps{\av{\frac{y}{r}}^2}\, d\mu(y)}.
\end{align}
We make the following claim. 
\begin{lemma} \label{l:III}
Let $\Omega$ be twice continuously differentiable. We have that
\begin{align*}
    \av{\frac{1}{r} \int \frac{\ip{III(\xi, x), \nu}}{r} \, \phi \ps{\av{\frac{y}{r}}^2} \, d\mu(y)} \lesssim \frac{|x|^2}{r^4} \int_{B(0,r)} \dist(y, L_Q)\, d\mu(y). 
\end{align*}
\end{lemma}
The proof of Lemma \ref{l:III} is long, and not incredibly pleasant. We postpone it to the Appendix. 

Assuming Lemma \ref{l:III}, we have completed the bound of $A_3$.

\subsubsection{Estimates on $B_{1,1}$} \label{sub:B11}
Recall that 
\begin{align*}
    B_{1,1}(x) = \frac{1}{r} \int \frac{K(-y)}{r}\, \phi' \ps{\av{\frac{y}{r}}^2} \, \left[ - \av{\frac{DK(-y) \cdot x + \frac{1}{2}x^T DK(\xi) x}{r}}^2 \right]\, d \mu(y). 
\end{align*}
Let us estimate $\left< B_{1,1}(x), \nu \right>$; it is immediate from Remark \ref{r:dist-dot} that $|\left< K(-y), \nu\right>| \lesssim \dist(y, L_Q)$. Moreover,
\begin{align} \label{eq:GK_estSO}
   \av{\frac{DK(-y) \cdot x + \frac{1}{2}x^T D^2K(\xi) x}{r}}^2  \lesssim \frac{1}{r^2}\left(|x|^2 + \frac{|x|^4}{|y|^2}\right)  \lesssim \frac{|x|^2}{r^2};
\end{align}
this follows from the bounds $|DK(y)\cdot x|\lesssim |x|$ and \begin{align*}
    |x^T D^2K(y)x | \lesssim \frac{|x|^2}{|y|};
\end{align*} 
also recall that 
since $x \in Q$ and $y \in A(0, r/2, 2r)$, where $r \sim A \ell(Q)$, and $A >0$ is a large constant, then $\frac{|x|}{|y|} \leq 1$, and so $\frac{|x|^4}{|y|^2} \leq |x|^2$. 
Finally, we see that 
\begin{align*}
    \av{ \left< B_{1,1}(x), \nu \right>}& \lesssim \frac{1}{r} \int \frac{\dist(y, L_{Q})}{r} \, \phi \ps{\av{\frac{y}{r}}^2} \, \frac{|x|^2}{r^2} \, d\mu(y) \\
    & \lesssim \frac{|x|^2}{r^4} \int_{B(0, r)} \dist(y, L_{Q}) \, d\mu(y). 
\end{align*}

\subsubsection{Estimates on $B_{1,2,2}$}
We want to compute 
\begin{align*}
    \av{\left< B_{1,2,2}(x), \nu\right> } = \av{\frac{1}{r} \int \frac{\left<K(-y), \nu \right>}{r} \phi' \ps{\av{\frac{y}{r}}^2} \, \left[ \left< \frac{K(-y)}{r}, \, \frac{\frac{1}{2}x^T D^2K(\xi) x}{r} \right> \right] \, d \mu(y)}.
\end{align*}
Again, $\av{\left< K(-y), \nu\right> } \lesssim \dist(y, L_Q)$. On the other hand,  using the same computation as for the estimate \eqref{eq:GK_estSO}, and in particular the bound $|D^2K(y)| \leq |y|^{-1}$, and also \eqref{e:xi-y-norm},  we see that
\begin{align} \label{eq:GK_estDP}
    \frac{1}{r^2}\av{\left< K(-y), x^T D^2K(\xi) x\right>} & \lesssim \frac{1}{r^2} \ps{ |y||x|^2 |y|^{-1} } = \frac{|x|^2}{r^2}.
\end{align}
Hence we obtain that
\begin{align*}
    \av{\left< B_{1,2,2}(x), \nu\right> } \lesssim \frac{|x|^4}{r^6} \int_{B(0,r)} \dist(y, L_Q) \, d\mu(y) \lesssim \frac{|x|^2}{r^4} \int_{B(0,r)} \dist(y, L_Q) \, d\mu(y).
\end{align*}

\subsubsection{Estimates for $B_2$} \label{sub:B2}
We want to compute $\av{\left<B_2(x), \nu \right>}$, that is
\begin{align*}
    & \frac{1}{r} \int \frac{\left<D K(-y) \cdot x, \nu \right>}{r}  \phi'\ps{\av{\frac{y}{r}}^2} \, \left[ SO\right] \, d\mu(y) \\
    & \enskip \enskip +  \frac{1}{r} \int \frac{\left< D K(-y) \cdot x, \nu \right>}{r}  \phi'\ps{\av{\frac{y}{r}}^2} \left[ DP \right] \, d\mu(y) \\
    & := B_{2,1}(x) + B_{2,2}(x). 
\end{align*}
Recall the definition of $[SO]$ in \eqref{e:SO} and that of $[DP]$ in \eqref{e:DP};
we first estimate the common term $\left< DK(-y) \cdot x, \nu \right>$. As before, we write it explicitly: 
\begin{align*}
    \left< DK(-y) \cdot x, \nu \right>= \left< D(|y|) \Omega(-y) \cdot x, \nu \right> + \left< |y| D\Omega(-y) \cdot x , \nu \right>.
\end{align*}
On one hand, we see that
\begin{align} \label{eq:GK_B21a}
    \left<D|y| \Omega(-y) \cdot x, \nu \right> = \left< D |y|, x \right>\left< \Omega(-y), \nu \right>.
\end{align}
Recall that, because $\nu$ is the normal unit vector to $K(L_Q)$, we can apply Remark \ref{r:dist-dot},  so to see that
\begin{align} \label{e:B2-1}
    |\left< \Omega(y), \nu \right>| \lesssim |y|^{-1} \dist(y, L_Q).
\end{align}
Moreover, 
\begin{align}\label{eq:GK_B21c}
    |\left< D|y|, x \right>| \lesssim |x|.
\end{align}
On the other hand, using once more Remark \ref{remark:BT_remdiff}, an in particular \eqref{e:der-sphere}, we see that $D\Omega(y) \cdot x= d_{\hat y} \Omega \cdot \hat y_\perp \frac{1}{|y|}\left< \hat y_\perp, x \right>$; and so, arguing as in \eqref{eq:SYM_bound_on_der},  
\begin{align} \label{eq:GK_B21b}
    \av{\left< D\Omega(y) \cdot x, \nu \right> }& = \frac{1}{|y|} \av{ \left< d_{\hat y} \Omega \cdot \hat y_\perp , \nu \right> } \av{ \left< x, \hat y_\perp \right>}
    = \frac{|x|}{|y|^2} \av{\left< d_{\hat y} \Omega \cdot \hat y_\perp , \nu \right> } \av{\left< y, e_z \right> } \nonumber\\
    & \lesssim  \frac{|x|}{|y|^2}  \av{\left< y, e_z \right> } 
    \lesssim \frac{|x|}{|y|^2} \dist(y, L_Q).
\end{align}
Putting together \eqref{eq:GK_B21a}, \eqref{e:B2-1} and \eqref{eq:GK_B21c} for one term, and \eqref{eq:GK_B21b} for the other, we see that
\begin{align}
    \av{\left< DK(-y)\cdot x, \nu \right>} \lesssim \frac{|x|}{|y|}\dist(y, L_Q) \leq \dist(y, L_Q). 
\end{align}
Thus the estimate \eqref{eq:GK_estSO} of the term $SO$ let us conclude that
\begin{align*} 
    \av{\left< B_{2,1}(x) ,  \nu \right>} & \lesssim \frac{1}{r^2} \int \av{\left< DK(y) \cdot x, \nu\right>} \phi' \ps{\av{\frac{y}{r}}^2} [SO] \, d\mu(y) \nonumber \\
    & \lesssim \frac{1}{r^2} \int  \dist(y, L_Q)  \phi' \ps{\av{\frac{y}{r}}^2} [SO] \, d\mu(y)\nonumber \\
    & \lesssim  \frac{1}{r^2} \int  \dist(y, L_Q)  \phi' \ps{\av{\frac{y}{r}}^2} \frac{|x|^2}{r^2} \, d\mu(y)\nonumber\\
    & \lesssim \frac{|x|^2}{r^4} \int_{B(0,r)} \dist(y, L_Q) \, d\mu(y).
\end{align*}
The same bound can be obtained for $|\left<B_{2,2}(x), \nu \right>$ by using the estimate \eqref{eq:GK_estDP} on the term $DP$. All in all we obtain
\begin{align}\label{eq:GK_boundB2}
    \av{\left< B_2(x), \nu \right>} \lesssim \frac{|x|^2}{r^4} \int_{B(0,r)} \dist(y, L_Q) \, d\mu(y).
\end{align}
\subsubsection{Estimates for $B_3$}
This term has the same form as $A_3$, the only difference being the presence of $DP$ and $SO$. However, from \eqref{eq:GK_estSO} and \eqref{eq:GK_estDP}, we see that both are $\leq 1$. Indeed recall that, for example, $|DP| \leq \frac{|x|^2}{r^2}$; but $x \in Q$, and $r \sim A \ell(Q)$, where $A$ is large. Hence $|DP| \leq 1$ and the same holds for $|SO|$. Thus we obtain the bound
\begin{align}\label{eq:GK_boundB3}
    |\left< B_3(x) , \nu \right> | \lesssim \frac{|x|^2}{r^4} \int_{B(0,r)} \dist(y, L_Q) \, d\mu(y)
\end{align}
as in Subsection \ref{subsection:A3}.

\subsubsection{Estimates for $C$}
 We Taylor expand $K(x-y)$ as before and split up $C(x)$ consequently so that
 \begin{align*}
     & C_1(x) = \frac{1}{r} \int \frac{K(-y)}{r} \phi''(\xi_{s_0,s}) [DP + SO]^2 d\mu(y)\\
     & C_2(x) = \frac{1}{r} \int \frac{DK(-y)}{r} \phi''(\xi_{s_0,s}) [DP + SO]^2 d\mu(y)\\
     & C_3(x) = \frac{1}{r} \int \frac{x^T D^2 K(\xi_{0,x}-y) x}{r} \phi''(\xi_{s_0,s}) [DP + SO ]^2 \, d\mu(y).
 \end{align*}
 The proofs that we gave for the previous terms hold for these ones, too. 
 Let us briefly sketch them in this situation. 
 \begin{itemize}
     \item The term $C_1$ can be estimated as it was done for $B_{1,1}$, see Subsection \ref{sub:B11}; indeed, it is immediate from Remark \ref{r:dist-dot} that
     \begin{align*}
         \av{ \left< C_1(x), \nu \right> } \lesssim \frac{1}{r^2} \dist(y, L_Q)  \phi''(\xi_{s_0,s}) \left[ DP + SO\right]^2 \, d\mu(y).
     \end{align*}
     Recall also that
     \begin{align*}
         \xi_{s_0,s} \in [s_0, s] = \left[ \av{\frac{y}{r}}^2, \av{\frac{x-y}{r}}^2\right];
     \end{align*}
     the definition of $\phi$ in \eqref{eq:SYM_smoothcut} implies therefore that $\phi''(\xi_{s_0,s}) \neq 0$ only if $y \in A\ps{0, \frac{r}{4}, 4r}$. Thus one can estimate $DP$ and $SO$ as it was done in \eqref{eq:GK_estSO} and \eqref{eq:GK_estSO}. This gives the desired bound
     \begin{align}\label{e:C1-bound}
         \av{\left< C_1(x), \nu \right> } \lesssim \frac{|x|^2}{r^4} \int_{B(0, 4r)} \dist(y, L_Q) \, d\mu(y). 
     \end{align}
     \item Let us now look at $C_2(x)$: this can be estimated like $B_2$, see Subsection \ref{sub:B2}. Note the similarities: both $C_2$ and $B_2$ have as main integrand $DK(y)\cdot x$, which can be estimated as in \eqref{eq:GK_B21a}, \eqref{e:B2-1}, \eqref{eq:GK_B21c} and \eqref{eq:GK_B21b}. Moreover, we have seen in the remark above that the presence of $\xi_{s_0,s}$ in the argument of the smooth cut off does not cause trouble, and therefore $SO$ and $DP$ have the same estimates as in \eqref{eq:GK_estSO} and \eqref{eq:GK_estDP}. 
     \item Finally $C_3$ can be estimated as $B_3$ (i.e. as $A_3$), see Subsection \ref{subsection:A3}. 
 \end{itemize}
  Hence we have the bound
 \begin{align} \label{eq:GK_boundC}
     |\left< C(x), \nu \right> | \lesssim \frac{|x|^2}{r^4} \int_{B(0,4r)} \dist(y, L_Q) \, d\mu(y),
 \end{align}
 These estimates together prove Lemma \ref{l:error-x}.
 
 \subsubsection{A further estimate on the error term}
 We now want to prove a bound on $E$ when evaluated at any point in $3Q$ (not just at a balanced point).
 \begin{lemma}
 Keep the notation as above; in particular $z \in 3Q$, $x_0, x_1$ are balanced points and $L_Q$ is the balanced line. Then 
 \begin{align} \label{l:error-z}
     |\left< E(z), \nu \right>| \lesssim \frac{|z|^2}{r^4} \int_{B(x_0,r)} \dist(y, L_Q) \, d\mu(y) + \frac{1}{A^2 r}\dist(z, L_{Q}). 
 \end{align}
 \end{lemma}
 \begin{proof}
 Again, we assume $x_0=0$.
 The proof of this lemma is very similar to the one for Lemma \ref{l:error-x}. 
 \begin{itemize}
     \item For the terms $B_{1,1}$, $B_{1,2,2}, C_1$ the proof goes through verbatim, since we only used that $|x_1| \leq r$, and the same holds for $|z|$ (recall that $r \in [A\ell(Q), 2A \ell(A)]$, and $A$ is a large constant).
     \item On the other hand, let us consider the term $\left< B_{2}(z), \nu \right>$:
\begin{align*}
    &\frac{1}{r} \int \frac{\left<D K(-y) \cdot z, \nu \right>}{r} \phi' \ps{ \av{\frac{y}{r}}^2}[SO] \, d \mu(y)\\
    & = \frac{1}{r} \int \frac{\left<(D |y|) \Omega(-y)\cdot z, \nu \right>}{r} \phi' \ps{\av{\frac{y}{r}}^2} \,[SO] d\mu(y) \\
    & + \frac{1}{r} \int \frac{\left< |y|D\Omega(y) \cdot z, \nu \right>}{r} \phi' \ps{\av{\frac{y}{r}}^2} [SO]\, d\mu(y).
\end{align*}
Now the first term is easily taken care of, as $\left< D|y| \Omega(y) \cdot z, \nu \right> = \left< D |y|, z \right>\left< \Omega(y), \nu \right>$.
On the other hand, recalling \eqref{e:der-sphere}, we have 
$$D \Omega(y) \cdot z = \frac{1}{|y|} \ip{z, \hat y_\perp} d_{\hat y} \Omega \cdot \hat y_\perp.$$
Denote by $\tilde z$ the projection of $z$ onto $L_Q$. We see that 
\begin{align*}
    |\ip{z, \hat y_\perp}| \leq |\ip{z- \tilde z, \hat y_\perp}| + |\ip{\tilde z, \hat y_\perp}|.
\end{align*}
 Now, it is clear that $|\ip{z-\tilde z, \hat y_\perp}| \leq |z- \tilde z| = \dist(z, L_Q)$. But also, 
 \begin{align*}
    & \av{ \left< \tilde z, \hat y_\perp \right> } = |\tilde z| \av{ \left< \frac{\tilde z}{|\tilde z|} , \hat y_\perp \right> } = |\tilde z| \av{ \left< \frac{x_1}{|x_1|} , \hat y_\perp \right> } \\
     & = \frac{|\tilde z|}{|x_1|} \av{ \left< x_1 , \hat y_\perp \right> } \lesssim  \av{ \left< x_1 , \hat y_\perp \right> },
 \end{align*}
 since $z \in 3Q$ and $\tilde z \in L_Q$, which is the span of $x_1$. Moreover, as in \eqref{eq:GK_B21b}, we have that $|\ip{x_1, \hat y_\perp}| \lesssim \dist(y, L_Q)$. With these considerations, and recalling that $|SO| \lesssim |z|^2 r^{-2}$ (see \eqref{eq:GK_estSO}), we obtain the bound
\begin{align*}
&\av{ \frac{1}{r} \int \frac{ \left< |y|D \Omega(y) \cdot z, \nu \right> }{r} \phi' \ps{\av{ \frac{y}{r}}^2} \,[SO]d\mu(y) }  \\& \lesssim  \frac{|z|^2}{r^4} \int_{B(0,r)} \dist(y, L_Q) \, d\mu(y) + \frac{1}{A^2 r} \dist(z, L_Q), 
\end{align*}
where in the last term we used the Ahlfors regularity of $\mu$.
    \item Let us sketch a proof of the bound for $|\left<A_3(z), \nu \right>|$; it is almost verbatim the same as that for $|\left<A_3(z), \nu \right>|$.
    First, recall that 
    \begin{align*}
        A_3(z) = \frac{1}{2r^2} \int z^T D^2(\xi_{z,0} -y) z \, \phi(|y|^2/r^2)\, d\mu(y).
    \end{align*}
    Put $\xi:=\xi_{z,0}-y$. As we did in Subsection \ref{subsection:A3}, we split the integrand $z^T D^2(\xi) z$ into three pieces, so to have
    \begin{align*}
        \left<z^T D^2K(\xi) z, \nu \right> = \left< I(\xi, z), \nu \right> + \left< II(\xi, z) ,\nu \right> + \left<III(\xi, z), \nu \right>.
    \end{align*}
    See \eqref{e:A3-3} and the expression above it. 
    First, we can see that
    \begin{align*}
        \left< I(\xi, z), \nu \right> = \frac{1}{|\xi|} \left< K(\xi), \nu \right> z^T D^2 (|\xi|) z. 
    \end{align*}
    Moreover, using Remark \ref{r:dist-dot}, we have, as in \eqref{e:A3-aI}, 
    \begin{align} \label{e:A3zI}
        |\left<K(\xi), \nu \right>| \lesssim |\left< \xi, e_z\right> | 
        & \leq |\left< \xi_{z, 0}, e_z \right>|+ |\left< y, e_z \right>|\nonumber \\
        & \lesssim |\left< \xi_{z,0}, e_z\right>| + \dist(y, L_Q).
    \end{align}
    With the second term we obtain an expression like \eqref{eq:SYMg_estI}. As for the term $|\left< \xi_{z,0}, e_z \right>|$, note first that because $\xi_{z,0} \in [0, z]$, i.e. the line connecting $0$ to $z$, we have that $|\left< \xi_{z,0}, e_z\right>| \leq |\left< z, e_z \right>|$. Moreover \begin{align*}
        |\left< z, e_z \right>| \lesssim \dist(z, L_Q).
    \end{align*}
    Recalling that $z \in 3Q$ (and so $|z|\sim r/A)$), that $|\xi|\sim r$ and the Ahlfors regularity of $\mu$, we then see that
    \begin{align} \label{e:A3zIb}
        & \frac{1}{2r^2} \int \frac{1}{|\xi|} \dist(z, L_Q) \av{z^T D^2(|\xi|)z} \phi(|y|^2/r^2) \, d\mu(y) \nonumber\\
        & \lesssim \frac{1}{r^2} \int_{B(0,r)} \frac{|z|^2}{|\xi|^2}\dist(z, L_Q) \, d\mu(y)
        \lesssim \frac{1}{A^2 r} \dist(z, L_Q).
    \end{align}
    This is as far as the term $I(\xi,z)$ is concerned. 
    
    Let us now look at $|\left< II(\xi,z), \nu \right>|$. Once again, the proof given in Subsection \ref{sss:A3-II} works here as well: the only difference is that one ends up with an extra term involving $\dist(z,L_Q)$ as for $\left<I(\xi, z), \nu \right>$. We spare the reader the details. 
    
    Finally, we look at the term $\left< III(\xi, z), \nu \right>$. Recall that this term was subsequently split into three further terms, which we called $A_{3,1}, A_{3,2}$ and $A_{3,3}$. 
    
    \underline{Estimate for $A_{3,1}$.} If we follow the same computation as for the corresponding term in Subsection \ref{sss:A3-III}, we end up considering again the quantity 
    \begin{align*}
        |\left< \Omega(\xi), \nu\right>|,
    \end{align*}
    as in \eqref{e:A3-III-A31a}. We can deal with this term as we did in \eqref{e:A3zI} and then obtain an estimate like \eqref{e:A3zIb}.
    
    \underline{Estimates for $A_{3,2}$.}
    We want to carry out the same computations that were carried out in the corresponding part of Subsection \ref{sss:A3-III}. Note however that to do so, we need to assume that $|z|\neq 0$  (in particular, see \eqref{e:A3-2}). This can be done without loss of generality: $\mu$ is Ahlfors regular, and thus $\mu(\{z\})=0$. We can then follow through the computations up to equation \eqref{e:A3-III-A32-a}, which in our present situation looks like
    \begin{align*}
        |\xi||\theta_\xi - \theta_z|.
    \end{align*}
    By triangle inequality, we have that
    \begin{align*}
        |\theta_\xi - \theta_z| \leq |\theta_\xi -\theta_{x_1}| + |\theta_{x_1} - \theta_z|. 
    \end{align*}
    With the second term, one can carry on as above, see \eqref{e:A3-III-A32-b} and the short paragraph above it. 
    We can deal with the first term as follows. First, note that by the assumption \eqref{eq:SYM_betasmall}, we see that 
    \begin{align*}
    |\theta_\xi - \theta_{x_1}| \lesssim \sin(|\theta_\xi - \theta_{x_1}|) \sim \frac{1}{|\xi|}\av{\left< \xi, e_z \right>}.
    \end{align*}
    But recall that $\xi=\xi_{z,0} -y$, and thus (as in \eqref{e:A3zI})
    \begin{align*}
        \av{\left< \xi, e_z \right>} \lesssim \av{\left< z, e_z\right> } + \dist(y, L_Q).
    \end{align*}
    One can then proceed as below \eqref{e:A3zI}. 
    \item The proof for the term $\left<B_3(z), \nu \right>$ is just the same as for the term $\left< B_3(x_1), \nu\right>$, taking also into account the slight modifications used above. 
    
    \item Finally, the term $\left< C(z) , \nu \right>$ may be split into $\left<C_1(z),\nu\right>$, $\left<C_2(z), \nu \right>$ and $\left<C_3(z), \nu\right>$; these ones may be dealt with as in the previous sections (bearing in mind the remarks made above). 
\end{itemize}
\end{proof}

\begin{lemma}
Let $\mu$ be an $\Omega$-symmetric measure. If $x_0$ and $x_1$ are balanced points of $Q$, with balanced plane $L_Q$, then,  for $z \in 3Q$, we have  
\begin{align}
   \frac{\dist(z, L_{Q})}{r} \lesssim \frac{\ell(Q)^2}{r^4} \int_{B(x_0, r)} \dist(y, L_Q) \, d\mu(y).
\end{align}
\end{lemma}
\begin{proof}
Again, let us assume that $x_0 = 0$. Since $\mu$ is an $\Omega$-symmetric measure and $0, x_1 \in \spt(\mu)$, we have that
\begin{align*}
    0= C_{\Omega, \phi}(0, r) - C_{\Omega,\phi}(x_1, r) = T(x_1) + E(x_1), 
\end{align*}
thus
\begin{align} \label{eq:GK_errT}
    - T(x_1) = E(x_1). 
\end{align}
Now, for $z \in  3Q$,
\begin{align*}
    \left< T(e_z), \nu \right> & = \left< T \ps{ \frac{z-\tilde z}{|z - \tilde z|}}, \nu \right>
 = \frac{1}{\dist(z, L_Q)} \left< T(z - \tilde z), \nu \right> \\
    & = \frac{1}{\dist(z, L_Q)} \left< T(z) - T(\tilde z), \nu \right>.
\end{align*}
Because $\tilde z \in L_Q$, and $\frac{x_1}{|x_1|}$ forms a basis for $L_Q$ we can write it as $\tilde z= \alpha \frac{x_1}{|x_1|}$; by the linearity of $T$ we then have that $T(\tilde z)= \frac{\alpha}{|x_1|} T(x_1)$, and, by \eqref{eq:GK_errT}, $\left<T(\tilde z), \nu\right>= \frac{\alpha}{|x_1|} \left< -E(x_1), \nu \right>$. Notice in passing that $\frac{|\alpha|}{|x_1|} \lesssim 1$, since $z \in 3Q$. Thus we see that
\begin{align*}
    - \left< T(e_z), \nu \right> = \frac{1}{\dist(y, L_Q)} \left( - \left<E(z), \nu\right> + \frac{\alpha}{|x_1|}\left< E(x_1), \nu\right> \right).
\end{align*}
By Lemma \ref{lemma:SYMg_lowboundT}, we see that $-\left< T(e_z), \nu\right> \gtrsim \frac{1}{r}$. Thus we have
\begin{align*}
    \frac{1}{r} \lesssim \frac{1}{\dist(y, L_Q)} \left( - \left<E(z), \nu\right> + \frac{\alpha}{|x_1|}\left< E(x_1), \nu\right> \right).
\end{align*}
Now we may take the absolute value of the right hand side of this inequality, and we apply the bounds proved for the error terms:
\begin{align*}
    \frac{1}{r} \lesssim \frac{1}{\dist(z, L_{Q})} \left(\frac{|x_1|^2}{r^4} \int_{B(0, r)} \dist(y, L_{Q}) \, d\mu(y) + \frac{1}{r} A^{-2} \dist(z, L_{Q})\right).
\end{align*}
As $A\geq 2$, we finally obtain
\begin{align*}
    \frac{\dist(z, L_Q)}{2r} \leq \frac{\dist(z, L_Q)}{r}(1-A^{-2}) \lesssim  \frac{\ell(Q)^2}{r^4} \int_{B(0, r)} \dist(y, L_{Q}) \, d\mu(y).
\end{align*}
\end{proof}
 
\begin{lemma} \label{lemma:GK_squaredist}
Keep the notation as above \textup{(}in particular $z \in 3Q$ is fixed and $N \sim \log(A)$\textup{)}. Then 
\begin{align}
    \left(\frac{\dist(z, L_{Q})}{r}\right)^2 \lesssim \frac{\ell(Q)^2}{r^2} \log(A) \sum_{k=0}^N \beta_{\mu, 2}(x_0, 2^k \ell(Q))^2. 
\end{align}
\end{lemma}
\begin{proof}
Let $x_0 \in F_0 \cap G_0$, where $F_0$ and $G_0$ are as defined in the proof of Lemma \ref{lemma:GK_goodpoints} and let us denote by $P(Q)$ the line which minimises $\beta_{\mu, 2}(Q)$, and by $P(x,r)$ the line which minimises $\beta_{\mu,2}(x,r)$, where $x \in \spt(\mu)$ and $r>0$; then we see that 
\begin{align*}
    \dist(y, L_{Q}) \leq \dist(y, P(x_0, r)) + \dist_H \left(P(x_0, r) \cap B(x_0, r), L_{Q} \cap B(x_0, r)\right).
\end{align*}
This, together with the assumption that $x_0$ belongs to $F_0$, let us conclude that 
\begin{align}
    \dist(y, L_{Q}) \lesssim r \sum_{k=0}^N \beta_{\mu, 2}(x_0, 2^k \ell(Q)),
\end{align}
where $N \sim \log(A)$. Hence we deduce that 
\begin{align*}
    \frac{1}{r} \int_{B(x_0,r)} \left( \frac{\dist(y, L_{Q})}{r} \right)^2 \, d\mu(y) & \lesssim \frac{1}{r} \int_{B(x_0, r)} \ps{ \frac{ \dist(y, L_{Q})}{r}}^2 d\mu(y) \\
    & + \frac{1}{r}\int_{B(x_0,r)} \ps{ \sum_{k=0}^N \beta_{\mu, 2} (x_0, 2^k \ell(Q)) }^2 d\mu(y)\\
    & \lesssim c_0\log(A)  \sum_{k=0}^N \beta_{\mu, 2} (x_0, 2^k \ell(Q))^2.
\end{align*}
\end{proof}
\section{Case when $\beta$'s are large}
In this section we will assume that \eqref{eq:SYM_betasmall} doesn't hold, and so that we have
\begin{align} \label{e:large-beta}
     \sum_{k=0}^N \beta_2(B(x_0, 2^k \ell(Q)) > \tau, 
\end{align}
where $N \sim \log(A)$, and $A \geq 1$ is a (large) constant to be fixed later on.
Let $\vr$ be a function so that 
\begin{align} 
& \vr \mbox{ is twice continuously differentiable } \label{eq:varphi1} \\
  &  1- \chara_{[0, 1]} \leq \vr \leq 1- \chara_{[0, \frac{1}{2}]}. \label{eq:varphi2}
\end{align}
Note that such $\vr$ will have the first derivative supported on $[\frac{1}{2}, 1]$. 
\begin{remark}
Let $\mu$ be an Ahlfors $1$-regular $\Omega$-symmetric measure in $\C$; if we define
\begin{align}
    R(x,r) = R_\mu^\Omega(x, r) := \int \frac{K(x-y)}{|x-y|^{2}} \, \vr
    \ps{\av{\frac{x-y}{r}}^2} d\mu(y),
\end{align}
then we see that 
\begin{align}\label{e:lb-R}
    R(x,r) = 0 \mbox{ for all } x \in \spt(\mu) \mbox{ and } r>0.
\end{align}
Indeed, note that the function
\begin{align*}
   \phi: r \mapsto \frac{\vr(r)}{|r|^2}  
\end{align*}
is bounded and satisfied $\lim_{r\to \infty}\frac{|\phi(r)|}{\mu(B(0,r)}=0$ since $\mu$ is Ahlfors regular. Thus we can apply Lemma \ref{lemma:SYM_borelequiv} and see that \eqref{e:lb-R} holds. 
\end{remark}

Let us fix some notation which will be used throught the sections below. Let $\vr$, $R_\mu^\Omega$ as above, and let $\mu$ be an Ahlfors 1-regular $\Omega$-symmetric measure in  $\C$; fix a $\mu$-cube $Q$ with side length $\ell(Q)$ and let $r \in [A \ell(Q), 2 A \ell(Q)]$, where $A$ is as above; denote by $L_Q$ the balanced line given in Lemma \ref{lemma:SYM_balancedcubes} and $x_0, x_1$ the two corresponding balanced points.

\begin{lemma} \label{lemma:GK_distboundRiesz}
 With the notation above \textup{(}in particular $r \sim \ell(Q)A$, $A>1$ a large constant\textup{)}, we have that for any $z \in 3Q$, 
\begin{align}
    \dist(z, L_Q) \lesssim \frac{\ell(Q)^2}{ r}
\end{align}
\end{lemma}
We will prove Lemma \ref{lemma:GK_distboundRiesz} in a few subsections. 
Without loss of generality, we will assume that $x_0=0$, and we will also set 
\begin{align*}
    x_1 = x \mbox{ and } e_x := \frac{x}{|x|}.
\end{align*}
Recall also the notation $e_z$ defined in \eqref{e:e_z}: if $z \in 3Q$ and $\tilde z$ is the orthogonal projection of $z$ onto $L_Q$, then we put
\begin{align*}
    e_z = \frac{z -\tilde z}{|z - \tilde z|}.
\end{align*}
The reader should also keep in mind the notation used in Remark \ref{remark:BT_remdiff} (in particular that of $\hat y$ and $\hat y_\perp$). 

As in Lemma \ref{lemma:GK_splitdiff}, we split the difference $R(0, r) - R(x, r)$ by Taylor expansion, where $x \in Q\setminus \eta Q$ belongs to the line $L_Q$ (here $\eta$ is as in \eqref{lemma:SYM_balancedcubes}). 
Note however that we do not expand the denominator $\frac{1}{|x-y|}$. 
Thus we obtain several terms,  with which we then construct $T(x)$ and $E(x)$ so that
\begin{align}
    0 = R(0, r)  - R(x, r) = T(x) + E(x).
\end{align}
See \eqref{eq:GK_T}, \eqref{eq:SYM_E} and the terms before for definitions; the reader should bear in mind however, that rather than having a weight of the form $\frac{1}{r^2}$ we now have one of the form $\frac{1}{|x-y|^2} \sim \frac{1}{|y|^2}$.

\subsection{Lower bounds on $\mathbf{T}$}
Recall from \eqref{eq:GK_T} that this is given by 
\begin{align}
    & T(x)= A_2(x) + B_{1,2,1}(x), \mbox{ where  }\\
    & A_2(x) = \int \frac{DK(-y) \cdot x}{|y|^2} \vr \ps{\av{\frac{y}{r}}^2} \, d\mu(y) \\
    & B_{1,2,1}(x) = \int \frac{K(-y)}{|y|^2} \vr'\ps{\av{\frac{y}{r}}^2} \left[ \left< \frac{K(y)}{r}, \frac{DK(y) \cdot x}{r} \right> \right] \, d\mu(y). \label{eq:SYM_largebeta_1}
\end{align}

\subsubsection{Bounds for $B_{1,2,1}$}

Note that for any vector $z \in \C$, we can split (as usual using \eqref{e:DK-expression}) the dot product in the square brackets in \eqref{eq:SYM_largebeta_1} as follows:
\begin{align}
    \left< K(y) , DK(y) \cdot z \right> = \left< K(y), \left<\hat y, z \right> \Omega(y) \right> + \left< K(y), \left<z, \hat y_\perp\right> d_{\hat y} \Omega \cdot \hat y_\perp \right>. 
\end{align}
Now, since $d_{\hat y}\Omega \cdot \hat y_\perp \in T_{\Omega(y)}\mathbb{S}$ while $\Omega(y)$ is exactly perpendicular to it, the second term vanishes; we are left with the first one, which equals to 
\begin{align}
    |y|^{-2} \left< y, z\right>|K(y)|^2 = \left< y, z\right>.
\end{align}
Thus we have that 
\begin{align*}
 & \left< B_{1,2,1}(e_x), \Omega(e_x) \right> + \left< B_{1,2,1}(e_z), \Omega(e_z)\right> \\
 & = \frac{1}{r^2}\int\left( \left< K(-y), \Omega(e_x)\right> \left< y, e_x \right> + \left< K(-y), \Omega(e_z)\right> \left<y, e_z\right> \right) \frac{\vr'\ps{\av{\frac{y}{r}}^2}}{|y|^2} \, d\mu(y).
\end{align*}

\begin{sublemma} \label{lemma:SYM_lowbound1}
\begin{align}\label{e:B121-large-beta}
    \frac{1}{r^2}\int\left( \left< K(-y), \Omega(e_x)\right> \left< y, e_x \right> + \left< K(-y), \Omega(e_z) \right> \left<y, e_z\right> \right) \frac{\vr'\ps{\av{\frac{y}{r}}^2}}{| y|^2} \, d\mu(y) \gtrsim \frac{1}{r}.
\end{align}
\end{sublemma}
\begin{proof}
The proof of this is the same as in Sublemma \ref{sublemma:SYMg_lowA2}, i.e. it mainly follows from Lemma \ref{l:dot-prod-sign}. Let us give a couple of remarks. First, even if $\Omega(e_x) \neq \nu$ (where $\nu$ is the unit normal to $K(L_Q)$), the proof given for Sublemma \eqref{sublemma:SYMg_lowA2} works for $\Omega(e_x)$ as well. Indeed the reason why we used $\nu$ in Section \ref{s:small-beta} is to obtain a bound in terms of $\dist(y, L_Q)$ for the error term - in that case we needed $\nu$, and $\Omega(e_z)$ would not work. Second, note that with for $\vr$ as in \eqref{eq:varphi1} and \eqref{eq:varphi2}, the derivative  $\vr'$ will be non negative and supported on an annulus $A\ps{0, \frac{1}{2}, 2}$. Thus $\vr'$ effectively has all the property of $\phi$ (as in \eqref{eq:SYM_smoothcut}) that were used in the proof of Sublemma \ref{sublemma:SYMg_lowA2}. 
\end{proof}

\subsubsection{Bounds for $A_2$} Keeping the notation as above, let us prove the following. 
\begin{sublemma}
\begin{align} \label{e:A2-large-beta}
    \ip{A_2(e_x), \Omega(e_x)} + \ip{A_2(e_z), \Omega(e_z)} \gtrsim \frac{1}{r^2}.
\end{align}
\end{sublemma}
\begin{proof}
For a unit vector $e \in \mathbb{S}$, we once again use \eqref{e:DK-expression} and split the dot product as
\begin{align*}
    \left< DK(y) \cdot e, \Omega(e) \right> = \left< \hat y, e\right>\left<\Omega(y), \Omega(e)\right> + \left< e, \hat y_\perp\right> \left< d_{\hat y} \Omega \cdot \hat y_\perp, \Omega(e)\right> .
\end{align*}
We also split the left hand side of \eqref{e:A2-large-beta} accordingly:
\begin{align*}
    &\int \left(\left<DK(y) \cdot e_x, \Omega(e_x)\right> + \left<DK(y) \cdot e_z, \Omega(e_z) \right>\right)  \frac{\vr\ps{\av{\frac{y}{r}}^2}}{|y|^2} \, d\mu(y)\\
    & = \int \ps{ \left< \hat y, e_x \right> \left< \Omega(y), \Omega(e_x)\right> + \left< \hat y, e_z \right> \left< \Omega(y), \Omega(e_z)\right>} \, \frac{\vr\ps{\av{\frac{y}{r}}^2}}{|y|^2} \,d\mu(y) \\
    & \enskip \enskip  + \int \ps{ \left< e_x, \hat y_\perp\right> \left< d_{\hat y} \Omega \cdot \hat y_\perp, \Omega(e_x)\right> + \left< e_z, \hat y_\perp\right> \left< d_{\hat y} \Omega \cdot \hat y_\perp, \Omega(e_z)\right> } \frac{\vr\ps{\av{\frac{y}{r}}^2}}{|y|^2} \,d\mu(y)\\
    & =: I + II.
\end{align*}
Let us first look at $I$. Recall that $\vr\ps{\frac{|\cdot|^2}{r^2}}$ is supported on $\C \setminus B(0, r/2)$. We further split $I$:
\begin{align*}
    I = \int_{A(0, r/2, 2r)} + \int_{\C \setminus B(0, 2r)} =: I_1 + I_2. 
\end{align*}
Now, $I_1$ is basically the same integral as the one appearing on the left hand side of \eqref{e:B121-large-beta}, and thus
\begin{align}
    I_1 \gtrsim  r. 
\end{align}
Also for $I_2$ the proof is similar to that of Sublemma \ref{sublemma:SYMg_lowA2}: first, we write
\begin{align*}
    I_2= \int_{\ps{ \C \setminus B(0,2r)} \cap G} +  \int_{\ps{ \C \setminus B(0,2r) }\cap H} +\int_{\ps{ \C \setminus B(0,2r)} \cap F},
\end{align*}
where $G$, $F$ and $H$ are defined as in \eqref{e:GHF}.
Now, when we integrate over $(\C \setminus B(0, 2r))\cap G$
we use Sublemma \ref{lemma:SYM_lowbound1}; hence we obtain
\begin{align*}
    & \int_{\ps{ \C \setminus B(0,2r)} \cap G}\ps{ \left< \hat y, e_x \right> \left< \Omega(y), \Omega(e_x)\right> + \left< y, e_z \right> \left< K(y), \Omega(e_z)\right>} \, \frac{\vr\ps{\av{\frac{y}{r}}^2}}{|y|^2} \,d\mu(y) \\
    & \gtrsim  \int_{\ps{ \C \setminus B(0,2r)} \cap G} \frac{1}{|y|^2} \, d\mu(y) \\
    & =  \sum_{k=0}^\infty \int_{A(0, 2^{k+1}r, 2^{k+2}r )\cap G} \frac{1}{|y|^2} \, d\mu(y).
\end{align*}
Similarly, using \eqref{e:A2-2}, we obtain that 
\begin{align*}
    & \int_{\ps{ \C \setminus B(0,2r) }\cap H} \ps{ \left< \hat y, e_x \right> \left< \Omega(y), \Omega(e_x)\right> + \left< y, e_z \right> \left< K(y), \Omega(e_z)\right>} \, \frac{\vr\ps{\av{\frac{y}{r}}^2}}{|y|^2} \,d\mu(y) \\
    & \gtrsim \sum_{k=0}^\infty \int_{A(0, 2^{k+1}r, 2^{k+2}r) \cap H} \frac{1}{|y|^2} \, d\mu(y), 
\end{align*}
and we can say the same for $\int_{\ps{ \C \setminus B(0,2r)} \cap F}$. 
Then we see that, because $\mu$ is Ahlfors regular with constant $C_0$, 
\begin{align*}
    I_2 & \gtrsim  \sum_{k=0}^\infty \int_{A(0, 2^{k+1}r, 2^{k+2}r )} \frac{1}{|y|^2} \, d\mu(y) 
     \gtrsim_{C_0} \frac{1}{r} \sum_{k=0}^\infty \frac{1}{2^k} = C(C_0) \frac{1}{r},
\end{align*}
and therefore
\begin{align}\label{e:A2-large-beta-I}
    I \gtrsim_{C_0} \frac{1}{r}.
\end{align}
We now take care of $II$, that is, the the integral
\begin{align} \label{eq:SYM_remainingintegral}
    \int \ps{|y| \left<\hat y_\perp, e_x \right>\left< d_{\hat y} \Omega \cdot \hat y_\perp , \Omega(e_x) \right> +  \left<\hat y_\perp, e_z \right>\left< d_{\hat y} \Omega \cdot \hat y_\perp , \Omega(e_z) \right>} \frac{\vr\ps{\av{\frac{y}{r}}^2}}{|y|^2}.
\end{align}
Let us first look at the term
\begin{align*}
    \left<\hat y_\perp, e_x \right>\left< d_{\hat y}\Omega \cdot \hat y_\perp, \Omega(e_x)\right>.
\end{align*}
By symmetry, we can assume that $0 \leq \arg(y)- \arg(e_x) \leq \pi/2$, as in \eqref{eq:SYMg_symmetry}. Now, as observed before $\left<\hat y_\perp, e_x\right> = - \left< \hat y, e_z\right> $; on the other hand,
    $d_{\hat y} \Omega \cdot \hat y_\perp \in T_{\Omega(y)}\mathbb{S}$, and thus, as in \eqref{e:A2-1},
\begin{align*}
    \left< d_{\hat y} \Omega \cdot \hat y_\perp , \Omega(e_x) \right> \sim  \left< \Omega(y)_\perp, \Omega(e_x) \right>,
\end{align*}
where we denoted by $\Omega(y)_\perp$ the vector $\Omega(y)$ rotated by $90$ degrees counterclockwise. 
We have
\begin{align*}
    \left< \Omega(y)_\perp, \Omega(e_x) \right> = - \left< \Omega(y), \Omega(e_x)_\perp\right> = - \left<\Omega(y),\nu \right>.
\end{align*}
Therefore, we see that
\begin{align*}
    \left< \hat y_\perp, e_x \right> \left< d_{\hat y}\Omega \cdot \hat y_\perp, \Omega(e_x)\right> \sim_{\delta_\Omega} \left< \hat y, e_z \right> \left< \Omega(y), \nu\right>.
\end{align*}
Similarly, if we now look at the second term in the integral \eqref{eq:SYM_remainingintegral}, we have
\begin{align*}
    \left<\hat y_\perp, e_z \right>\left< d_{\hat y} \Omega \cdot \hat y_\perp , \Omega(e_z) \right>\sim_{\delta_\Omega}  \left< \hat y_\perp, e_z \right> \left< \Omega(y)_\perp, \Omega(e_z) \right>.
\end{align*}
Once again, we split the domain of integration into subsets $G$, $F$ and $H$, as defined in \eqref{e:GHF} and we argue as for $I$. Thus
\begin{align*}
    II \gtrsim \frac{1}{r}.
\end{align*}
This together with the estimate for $I$, \eqref{e:A2-large-beta-I}, proves the sublemma.
\end{proof}

We now have that 
\begin{align} \label{eq:SYM_lowboundT}
   |\ip{T(e_x), \Omega(e_x)} + \ip{T(e_z), \Omega(e_z)}| \gtrsim  r^{-1}. 
\end{align}

\subsection{Upper bounds on the error term $\mathbf{E}$}
\begin{lemma} 
Keep the notation as in the preceding subsection. Then
\begin{align}
    |E(x)| \lesssim \frac{|x|^2}{r^2} \sim \frac{\ell(Q)^2}{r^2}.
\end{align}
\end{lemma}
\begin{proof}
The proof of this is similar to the one given for Lemma \ref{l:error-x} | actually it's easier: we do not need to bound in terms of the distance to $L_Q$, but only in terms of absolute values.
Let us give a sketch of the proofs. Recall from \eqref{eq:SYM_E} that 
\begin{align*}
    E(x) = A_3(x) + B_{1,1}(x) + B_{1,2,2}(x) + B_2(x) + B_3(x) + C(x).
\end{align*}
\begin{itemize}
    \item The term $A_3$ can be bounded as follows.
    \begin{align*}
        \av{ \int \frac{\frac{1}{2} x^T D^2K(\xi) x }{|x-y|^2} \, \vr\ps{\av{\frac{y}{r}}^2} \, d\mu(y) } & \lesssim \int \frac{|x|^2}{|y|^3}\,\vr\ps{\av{\frac{y}{r}}^2} \, d\mu(y)\\
        & \lesssim \frac{|x|^2}{r} \int \frac{1}{|y|^2} \vr\ps{\av{\frac{y}{r}}^2} \, d\mu(y).
    \end{align*}
    Note that 
    \begin{align*}
        \int \frac{1}{|y|^2} \vr\ps{\av{\frac{y}{r}}^2} \, d\mu(y) \lesssim \frac{1}{r},
    \end{align*}
    since $\mu$ is Ahlfors regulars. Thus we obtain that
    \begin{align*}
        |A_3(x)| \lesssim \frac{|x|^2}{r^2}.
    \end{align*}
    \item Let us look at $|B_{1,1}(x)|$. In this instance, we are integrating against $\vr'$, which is supported on the annulus $A(0, 1/2,2)$. Thus 
    \begin{align*}
        |B_{1,1}(x)| \lesssim \frac{1}{r^2} \int |K(y)| |SO|  \vr' \ps{\av{\frac{y}{r}}^2} \, d\mu(y).
    \end{align*}
    Recall from \eqref{eq:GK_estSO} that $|SO| \lesssim \frac{|x|^2}{r^2}$; this estimate is still valid here since $\spt(\vr')$ is basically the same as $\spt(\phi)$ (as in \eqref{eq:SYM_smoothcut}). It is immediate to see then that
    \begin{align*}
        |B_{1,1}(x)| \lesssim \frac{|x|^2}{r^2}. 
    \end{align*}
    \item The estimate for $|B_{1,2,2}(x)|$ is almost exactly the same. We just need to recall that $|D^2K(\xi_{0,x}-y)| \lesssim \frac{1}{|\xi_{0,x}-y|} \sim \frac{1}{|y|}\sim \frac{1}{r}$.
    \item The estimates for the remaining terms are just as immediate. 
\end{itemize}
\end{proof}

We are now ready to prove Lemma \ref{lemma:GK_distboundRiesz}.

\begin{proof}[Proof of Lemma \ref{lemma:GK_distboundRiesz}]
We see that
\begin{align}
    |\ip{T(e_x), \Omega(e_x)}| \leq |T(e_x)| \lesssim \frac{1}{\ell(Q)} |T(x)| \leq \frac{1}{\ell(Q)} |E(x)| \lesssim \frac{\ell(Q)}{r^2}.
\end{align}
On the other hand, we have
\begin{align}
    |T(e_z) \cdot \nu| \leq |T(e_z)| = \frac{1}{\dist(z, L_Q)} T(z-\tilde z) \lesssim \frac{1}{\dist(z, L_Q)} \frac{\ell(Q)^2}{r^2}.
\end{align}
Now, notice that $\dist(z, L_Q) \lesssim \ell(Q)$. Thus, also using \eqref{eq:SYM_lowboundT}, 
\begin{align*}
     \frac{1}{r} & \lesssim |\left<T(e_x), \Omega(e_x)\right> +\left< T(e_z),\Omega(e_z)\right>| 
     \leq |T(e_x)| + |T(e_z)| 
    \lesssim \frac{\ell(Q)}{r^2} + |T(e_z)| \\
    & \lesssim \frac{1}{\dist(z, L_Q)}\frac{\ell(Q)^2}{r^2}.
\end{align*}
Thus
\begin{align} \label{eq:SYM_distbound1}
    \dist(z, L_Q) \leq C  \frac{\ell(Q)^2}{r}.
\end{align}
\end{proof}

\section{The measure $\mu$ is flat}
In this section we will show that $\mu$ lies in a line, from which we derive that $\mu$ is, in fact, flat. Recall the notation \eqref{e:betaQ}.
\begin{lemma}
Let $Q \in \mucubes$ be a cube so that $ \ell(Q) \sim \frac{1}{A} r$, where $A>1$ is a (possibly large) constant. Then we have
\begin{align} \label{eq:SYM_betabound}
    \beta_2(Q)^2 \mu(Q) \lesssim \frac{\ell(Q)^2}{r^2 C(\tau )} \log(A) \sum_{Q \subset P \subset \widehat Q} \beta_2(P)^2 \mu(Q),
\end{align}
where $\widehat Q \in \mucubes$ is the unique cube such that $Q \subset \widehat Q$ and $\ell(\widehat Q) \sim A\ell(Q)$. 
\end{lemma}
\begin{proof}
Suppose first that $x_0$ is so that $\sum_{k=0}^N \beta(x_0, \ell(Q)2^k) < \tau$ (where $N \sim \log(A)$). Then \eqref{eq:SYM_betabound} follows  from Lemma \ref{lemma:GK_squaredist} and the fact that 
\begin{align}
\sum_{k=0}^N \beta(x_0, \ell(Q)2^k)^2 \lesssim \sum_{Q \subset P \subset \widehat Q} \beta(P)^2
\end{align}

On the other hand, suppose that $x_0$ is so that $\sum_{k=0}^N \beta_2(x_0, \ell(Q)2^k) > \tau$. Then using Lemma \ref{lemma:GK_distboundRiesz},  we see that
\begin{align} 
    \frac{\dist(z, L_Q)^2}{\ell(Q)^2} \lesssim \frac{\ell(Q)^2}{r^2 \, \tau^2} \log(A) \sum_{Q \subset P \subset \widehat Q}\beta_2(P)^2.
\end{align}
This proves the lemma also in this case. 
\end{proof}

\begin{lemma}
For $A$ large enough, and $\delta_0, \tau$ chosen appropriately, the support of $\mu$ lies in a line $L$.
\end{lemma}
\begin{proof}
We see that \eqref{eq:SYM_betabound} implies that for some cube $S \in \mucubes$, and $1>\gamma >0$,
\begin{align} \label{e:flat1}
    \sum_{Q \subset S} \frac{\beta_2(Q)^2 \mu(Q)}{\ell(Q)^{1+\gamma}} \leq C \frac{\log(A)}{A^2 } \sum_{Q \subset S} \sum_{Q \subset P \subset \widehat Q} \beta_2(P)^2 \frac{\mu(Q)}{\ell(Q)^{1+\gamma}}.
\end{align}
Let $N_A \in \N$ be so that $A \sim 2^{N_A}$ and pick $\gamma<1$ sufficiently small.  
Using the Ahlfors regularity of $\mu$, we see that
\begin{align} \label{e:flat2}
    \sum_{\substack{Q \subset S\\ \widehat Q \subset S}} \sum_{Q \subset P \subset \widehat Q}  \beta_2(P)^2 \frac{\mu(Q)}{\ell(Q)^{1+\gamma}} & = \sum_{P \subset S} \beta_2(P)^2 \sum_{\substack{Q \subset P\\\hat Q \subset S\\ |\log(\ell(Q)) - \log(\ell(P))| \leq  2 A}} \frac{\mu(Q)}{\ell(Q)^{1+\gamma}} \nonumber \\
    & \lesssim_{C_0} \sum_{P \subset S} \beta_2(P)^2 \sum_{k=0}^{N_A} \sum_{\substack{Q \subset P \\ \hat Q \subset S \\ \ell(Q) \sim 2^{-k} \ell(P)}} \frac{1}{\ell(Q)^\gamma}\nonumber\\
    & \lesssim_{C_0} \sum_{P \subset S} \beta_2(P)^2 \frac{1}{\ell(P)^\gamma}\sum_{k=0}^{N_A} 2^{k(1+\gamma)}\nonumber\\
    & \lesssim_{C_0} A \log(A) \sum_{P \subset S} \beta_2(P)^2 \frac{1}{\ell(P)^\gamma}  \nonumber \\
   &  \sim_{C_0} A \log(A) \sum_{P \subset S} \beta_2(P)^2 \frac{\mu(P)}{\ell(P)^{1+\gamma}} .
\end{align}
On the other hand, using the trivial bound $\beta_2(P)^2 \lesssim 1$, we have that
\begin{align} \label{e:flat3}
    \sum_{\substack{Q \subset S\\ \widehat Q \supset S}} \beta_2(P)^2 \ell(Q)^{-\gamma} & \lesssim \sum_{\substack{Q \subset S\\ \widehat Q \supset S}} \frac{1}{\ell(Q)^\gamma} \sum_{Q \subset P \subset \widehat Q} 1 \nonumber \\
    & \lesssim \log(A) \sum_{\substack{Q \subset S\\ \widehat Q \supset S}} \frac{1}{\ell(Q)^\gamma} \nonumber \\
    & \sim \log(A) \sum_{k=0}^{N_A} \sum_{\substack{Q \subset S\nonumber \\ \ell(Q) \sim 2^{-k}\ell(S)}} \frac{1}{\ell(Q)^\gamma}\nonumber \\
    & \lesssim \log(A) \frac{1}{\ell(S)^\gamma}\sum_{k=0}^{N_A} 2^{k(1+\gamma)} \nonumber\\
    & \lesssim \frac{\log(A)^2 A}{\ell(S)^\gamma}.
\end{align}
Hence, writing out the sum on the right hand side of \eqref{e:flat1} as sum of the two sums in \eqref{e:flat2} and \eqref{e:flat3}, we obtain
\begin{align}
    \ps{1 - C\frac{\log(A)^2}{A }} \sum_{Q \subset S} \frac{\beta_2(Q)^2 \mu(Q)}{\ell(Q)^{1+\gamma}} \leq C \frac{A \log(A)^2}{\ell(S)^\gamma}. 
\end{align}
An appropriate choice of $A>1$ gives 
\begin{align*}
    \sum_{Q \subset S} \beta_2(Q)^2 \ell(Q)^{-\gamma} \lesssim_A \frac{1}{\ell(S)^\gamma}.
\end{align*}
Sending $\ell(S) \to \infty$ gives the result.
\end{proof}

A lemma similar to the one below can be found in \cite{huovinen} and in \cite{mat95}. We add a proof for completeness.
\begin{lemma}
Let $\mu$ be an Ahlfors 1-regular, $\Omega$-symmetric measure whose support is contained in a line $L$. Then 
\begin{align}\label{e:line-support}
    \spt(\mu) = L.
\end{align}
\end{lemma}
\begin{proof}
Suppose that $\spt(\mu) \neq L$. Since $\mu$ and $\hm{1}|_{L}$ are mutually absolutely continuous, there exists a segment $L' \subset L$ so that $L' \cap \spt(\mu) = \emptyset$ and we can also choose it so that $\dist(L', \spt(\mu)) = 0$. Let $a \in \spt(\mu)$ be the closest point to $L'$, and set $r:= \hm{1}(L') / 2$. Then it is clear that
\begin{align*}
    \int_{B(a, r)} K(a-y) \, d\mu(y) = \int_{B(a,r)} |a-y|\, \Omega(a-y)\, d\mu(y) \neq 0.
\end{align*}
This contradicts the $\Omega$-symmetricity of $\mu$. 
\end{proof}

Again, the next Lemma can be found in \cite{mat95}. We add the proof for completeness. 
\begin{lemma}
Let $\mu$ be an $\Omega$-symmetric, Ahlfors 1-regular measure such that $\spt(\mu)= L$, for some line in $\C$. Then there exists a constant $c>0$ such that   
\begin{align}
    \mu= c \hm{1}|_{L}.
\end{align}
\end{lemma}
\begin{proof}
Let $x_0 \in L$ and consider an annulus $A(x_0, R-\epsilon, R+\epsilon)$. Let $x_1$ and $x_2$ be the points where $L$ intersects $\partial B(x_0, R)$. Then, using Lemma \ref{lemma:SYM_borelequiv}, we see that 
\begin{align*}
    0 & = \int_{A(x_0, R-\epsilon, R+ \epsilon) } \frac{K(x_0- y)}{|x_0-y|} \, d\mu(y)\\
    & = \int_{B(x_1, \epsilon)} \frac{K(x_0- y)}{|x_0-y|} \, d\mu(y) +\int_{B(x_2, \epsilon)} \frac{K(x_0- y)}{|x_0-y|} \, d\mu(y). 
\end{align*}
Note that for any $y \in B(x_1, \epsilon) \cap \spt(\mu) = B(x_1, \epsilon) \cap L$, we have that $K(x_0-y)/|x_0-y| = \Omega(x_0-y) = \Omega(x_1-x_0)$, and similarly for $x_2$; since $\Omega$ is odd, we have $\Omega(x_0-x_2) = - \Omega(x_0-x_1)$. Hence we have that 
\begin{align*}
    0 = \Omega(x_0- x_1)\ps{ \mu(B(x_1, \epsilon))- \mu(B(x_2, \epsilon))}.
\end{align*}
Since $x_0$ and $R>0$ where arbitrary, we see $\mu$ gives the same measure to any ball of the same radius, thus is translation invariant, hence it must be a constant multiple of $\hm{1}|_{L}$. 
\end{proof}

\section{A simple application to singular integrals}
In this section we apply Theorem \ref{theorem:main}, to prove the following, and hence to prove Corollary \ref{cor:singular-integrals}.

\begin{theorem} \label{t:singular-integrals}
Let $K$ be as in the statement of Theorem \ref{theorem:main} and let $\mu$ be a Radon measure in the plane such that for $\mu$-almost all $x \in \C$ the following conditions hold. 
\begin{itemize}
    \item The lower density is positive and the upper density is finite, that is 
\begin{align}
     & \theta^{1,*}(\mu,x) := \limsup_{r \to 0} \frac{\mu(B(x,r)}{r} < + \infty; \label{e:up-density}\\
     & \theta_{*}^1(\mu, x) := \liminf_{r \to 0} \frac{\mu(B(x,r))}{r} >0. \label{e:low-density}
\end{align}
\item The principal value 
\begin{align} \label{e:pv}
    \lim_{\epsilon \downarrow 0} \int_{\C \setminus B(x,\epsilon)} \frac{K(x-y)}{|x-y|^2} \, d\mu(y) 
\end{align}
exists and is finite.
Then $\mu$ is $1$-rectifiable.
\end{itemize}
\end{theorem}

To prove this we will follow Mattila (see \cite{mat95}) in using Preiss' Theorem, which states that if a measure $\mu$ satisfies certain density conditions, and moreover its \textit{tangent measures} are all flat, then $\mu$ is rectifiable. Before stating this precisely, let us remind the reader what tangent measures are.

For a measure $\mu$ on $\C$, and a map $F: \C \to \C$, we denote the push forward measure under $F$ as $F[\mu]$.
\begin{definition}
A locally finite Borel measure $\nu$ is called a tangent measure of $\mu$ at $z_0 \in \C$, if it is a weak limit of some sequence $c_i T_{z_0, r_i} [\mu]$, where
\begin{align*}
    & 0 < c_i < \infty, \, r_i \downarrow 0 \mbox{ and } \\
    & T_{z_0, r_i}(A) = \mu(r_i A + z_0).
\end{align*}
\end{definition}
We will denote the set of tangent measures of $\mu$ at $z_0$ by $\Tan(\mu, z_0)$. We also put $\Tan(\mu) := \cup_{z_0 \in \spt(\mu)}\Tan(\mu, z_0)$.

\begin{remark}
Tangent measures were originally introduced by Preiss in \cite{preiss}; in this work, Preiss completed a theory started from Besicovitch in the 1920's which had as fundamental question the following: what is the relation  between limits of density ratios (such as \eqref{e:low-density} and \eqref{e:up-density}) and rectifiability? We refer the interested reader to the very readable monograph \cite{delellis} of De Lellis on Preiss' result. 
\end{remark}

\begin{theorem} 
[\cite{preiss}, Theorem ]\label{t:preiss}
Let $\mu$ be a locally finite Borel measure such that, for $\mu$-almost all $z_0 \in \C$, it holds that
\begin{align*}
    & 0< \theta_*^1(\mu, z_0) <+ \infty,
\end{align*}
and for all $\nu \in \Tan(\mu, z_0)$, $\nu$ is flat, or, in other words, 
\begin{align*}
    \nu = c \mathcal{H}^1|_{L}, \mbox{ where } L \mbox{ is a line.}
\end{align*}
Then $\mu$ is $1$-rectifiable. 
\end{theorem}
\begin{remark}
Preiss' theorem holds much more generally, that is, it holds for measures in $\R^n$ satisfying the  $d$-dimensional version of the density condition above. Moreover, Mattila proves in \cite{mat95}, Theorem 4.19, that for the case of Radon measures in the plane, assuming lower density positive is enough in the above theorem.
\end{remark}

With Preiss' Theorem at hand, our tasks are clear: if we show that any tangent measure $\nu$ (of a Borel measure $\mu$ which satisfies \eqref{e:up-density}, \eqref{e:low-density} and \eqref{e:pv}) is $\Omega$-symmetric and Ahlfors $1$-regular, then we are done.

\begin{lemma} \label{l:tangent-sym}
Let $\nu  \in \Tan(\mu, z_0)$, where $\mu$ (and thus $K$) satisfies the hypotheses of Theorem \ref{t:singular-integrals}. Then $\nu$ is $\Omega$-symmetric.
\end{lemma}

The following are standard properties of tangent measures. We will need them to prove Lemma \ref{l:tangent-sym}.
\begin{lemma}[\cite{preiss}, Theorem 2.5] \label{l:Tan-non-empty}
Let $\mu$ be a Radon measure satisfying \eqref{e:up-density} and \eqref{e:low-density} for $\mu$-almost all $x \in \C$. Then
\begin{align} \label{e:tan-not-empty}
    \Tan(\mu, x) \neq  \emptyset. 
\end{align}
Moreover, if $\nu \in \Tan(\mu, x)$, then 
\begin{align} \label{e:0-in-spt}
    0 \in \spt(\nu).
\end{align}
\end{lemma}
For a quick proof of \eqref{e:tan-not-empty} under the additional hypothesis that $$\limsup_{r \downarrow 0} \frac{\mu(B(x, 2r))}{\mu(B(x,r))}< \infty,$$ see \cite{mattila}, Theorem 14.3. Note that this assumption is satisfied in our case since $\mu$ satisfies \eqref{e:up-density} and \eqref{e:low-density}.

Under this same assumption, one can easily see that \eqref{e:0-in-spt} holds. We refer the reader to (2) on page 187 of \cite{mattila}. 

The following fact will also be useful. If a measure $\mu$ satisfies \eqref{e:up-density} and \eqref{e:low-density}, then there is a sequence $r_i \downarrow 0$, so that
\begin{align} \label{e:tan-specific-form}
    \nu = c \, \lim_{i \to \infty } \frac{T_{x, r_i}}{r_i} [\mu],
\end{align}
where $c$ is some positive number. See (3), page 187 in \cite{mattila}. 

We will need the following fact, see (3), page 16 in \cite{mat95}. Suppose that $\mu_i \to \nu$ weakly, where $\mu_i$ and $\mu$ are all locally finite borel measures on $\C$; let $z_i$ be a sequence of points in $\C$ converging to $z$ and consider two radii $0<r<R<\infty$. If $\nu(\partial(B(z,r))= \nu(\partial(B(z,R))=0$, and $\vr_i$ is a sequence of continuous function converging uniformly to $\vr$, then
\begin{align} \label{e:weak-fact}
    \lim_{i \to \infty} \int_{B(z_i, R)\setminus B(z_i, r)} \vr_i \, d\mu_i = \int_{B(z,R) \setminus B(x,r)} \vr\, d\mu. 
\end{align}

We will also need the following fact. See \cite{mattila}, Lemma 14.7 for a proof. 
\begin{lemma}\label{l:ADR-tangents}
Let $\mu$ be a Radon measure on $\C$ satisfying \eqref{e:up-density} and \eqref{e:low-density}. Then for $\mu$-almost all point $x \in \C$, if $\nu \in \Tan(\mu,x)$, then $\nu$ is Ahlfors $1$-regular, with regularity constant depending on the upper and lower densities.
\end{lemma}

Let us now prove the following. 
\begin{sublemma} \label{sl:tan-0-sym}
Let $\mu$ be a Borel measure on $\C$ satisfying the hypotheses of Theorem \ref{t:singular-integrals}. Then for $\mu$-almost every $x \in \C$, if $\nu \in \Tan(\mu, x)$, $\nu$ is $\Omega$-symmetric at $0$.
\end{sublemma}
\begin{proof}
Note first that the existence of the principal value at $x \in \spt(\mu)$, i.e. \eqref{e:pv}, implies that for $0<r<R$, 
\begin{align} \label{e:pv2}
    \lim_{R \downarrow 0} \int_{B(x, R) \setminus B(x,r)} \frac{K(x-y)}{|x-y|^2} \, d\mu(y) = 0. 
\end{align}
For the sake of completeness, let us sketch a proof of this. Suppose that there exists sequences $R_i \downarrow 0$ and $r_i \downarrow 0$ with $r_i< R_i$ for all $i \in \N$, so that 
\begin{align} \label{e:div-series}
    \liminf_{i \to \infty} \av{ \int_{B(x,R_i) \setminus B(x,r_i)} \frac{K(x-y)}{|x-y|^2} \, d\mu(y) } \geq \epsilon_0,
\end{align}
for some positive number $\epsilon_0>0$. But then one can choose a sequence $\epsilon_i \downarrow 0$ such that
\begin{align*}
    \lim_{i \to \infty} \int_{\C \setminus B(x,\epsilon_i)} & = \int_{\C \setminus B(x,1)} + \lim_{k \to \infty} \sum_{i = 0}^k \int_{B(x,R_i) \setminus B(x,r_i)} \\
    & = \int_{\C \setminus B(x,1)} + \sum_{i=0}^\infty \int_{B(x, R_i) \setminus B(x,r_i)}. 
\end{align*}
But by \eqref{e:div-series}, this series diverges. Thus \eqref{e:pv2} holds. 

By Egorov's Theorem, it is enough to consider a compact subset $F \subset \spt(\mu)$ such that the convergence in \eqref{e:pv2} is uniform. Let $z_0 \in F$ be a point of density, and let $\nu \in \Tan(\mu, z_0)$; note that by Lemma \ref{l:ADR-tangents} we can assume that $\nu$ is Ahlfors regular; this in turn implies that whenever $x \in \spt(\nu)$, then the radii $s>0$ for which  $\nu(\partial B(x,s)) \neq 0$ must be finitely many. Then we see that for any $0<r<R< +\infty$ (but finitely many), we have that
\begin{align*}
    \int_{B(0,R)\setminus B(0,r)} \frac{K(y)}{|y|^2} \, d\nu(y) & \eqt{\eqref{e:weak-fact}} \lim_{i \to \infty} \int_{B(0,R) \setminus B(0,r)} \frac{K(y)}{|y|^2} \, \frac{T_{z_0, r_i}}{r_i}[\mu](y) \\
    & = \lim_{i \to \infty} \int_{B(z_0, r_i R) \setminus B(z_0, r_i r)} \frac{K(z_0-y)}{|z_0-y|^2}\, d\mu(y) = 0. 
\end{align*}
Then, using approximation, we can conclude that $\nu$ is $\Omega$-symmetric at $0$.
\end{proof}

The next well known property of tangents measures will let us conclude that $\nu \in \Tan(x,\mu)$ is actually an $\Omega$-symmetric Ahlfors regular measure.
\begin{lemma}[\cite{mattila},Theorem 14.16] \label{l:translation}
Let $\mu$ be a Radon measure on $\C$ which satisfies the density conditions \eqref{e:up-density} and \eqref{e:low-density} for $\mu$-almost every  $x \in \C$. Let $\nu \in \Tan(\mu, x)$. Then if $y \in \spt(\nu)$, then 
\begin{align*}
    T_{y, 1}[\mu] \in \Tan(\mu, x).
\end{align*}
\end{lemma}

\begin{proof}[Proof of Lemma \ref{l:tangent-sym}]
Let $\nu \in \Tan(\mu, x_0)$. By Lemma \ref{l:translation}, if $x \in \spt(\nu)$, then $T_{x, 1}[\nu] \in \Tan(\mu, x)$. Also, by Sublemma \ref{sl:tan-0-sym}, $0$ is a point of symmetry of $T_{x, 1}[\nu]$. Thus we see that, for $r>0$, 
\begin{align*}
    0 & = \int_{B(0, r)} \frac{K(y)}{|y|^2} \, d T_{x, 1}[\nu](z)
     = \int_{B(x,r)} \frac{K(x-y)}{|x-y|^2} \, d\nu(y).
\end{align*}
Hence $\nu$ is a symmetric measure. 
\end{proof}

\begin{proof}[Proof of Theorem \ref{t:singular-integrals}]
With the assumptions of the statement of the theorem, we see by Lemma \ref{l:tangent-sym} that all tangent measures of $\mu$ are $\Omega$-symmetric. Moreover, by Lemma \ref{l:ADR-tangents}, each of them is Ahlfors regular. Thus, any $\nu \in \Tan(\mu)$ satisfies the hypotheses of Theorem \ref{theorem:main}, and therefore we can conclude that 
\begin{align*}
    \nu = c \mathcal{H}^1|_L
\end{align*}
whenever $\nu \in \Tan(\mu)$, where $c>0$ and $L$ is a line. Hence Preiss' Theorem (Theorem  \ref{t:preiss}) let us conclude that $\mu$ is $1$-rectifiable. \end{proof}
\section{Appendix}
In this section, we prove Lemma \ref{l:III}.

For this computation, we will distinguish between real and imaginary part, rather than using the indices $1$ and $2$; also, we will do the computations using $y$ as variable (instead of $\xi$), to make the notation less cumbersome.

\begin{remark}
In this subsection we will use the complex argument function defined as 
$\arg(y) = \mathrm{atan2}(y_1, y_2)$, where $\mathrm{atan2}$ takes values in $[0, 2\pi]$ and is defined as
\begin{align}
    \mathrm{atan2}(y_1, y_2) := \begin{cases}
    \arctan\ps{\frac{y_2}{y_1}} \mbox{ when } y_1>0, y_2>0;
    \\
    \arctan\ps{\frac{y_1}{y_2}} + \pi \mbox{ when } y_1< 0, y_2 >0;\\
         \arctan\ps{\frac{y_1}{y_2}} + \pi \mbox{ when } y_1<0, y_2<0;\\
         \arctan\ps{ \frac{y_1}{y_2}} + 2 \pi \mbox{ when } y_1>0, y_2<0;\\
    0 \mbox{ when } y_1>0, y_2 = 0;\\
    \pi \mbox{ when } y_1<0, y_2 = 0;\\
    \frac{3}{2} \pi \mbox{ when } y_1=0, y_2<0;\\
    \mathrm{undefined when } y_1=0, y_2=0.
    \end{cases}
\end{align}
The fact that $\arg$ is undefined at the origin does not concern us, since we are integrating over an annulus. On the other hand, 
note that $\mathrm{atan2}$ has a discontinuity  on $(0, \infty) \subset \C$; this could cause some trouble since 
we will be taking derivatives of $\arg$.
But this won't be the case; indeed if $y$ lies in $(0, +\infty)$, we can reflect it and consider $-y$ instead, since $\Omega$ is odd. Thus we will assume, without loss of generality, that $y \notin (-\infty, 0)$ and hence taking derivative will make sense. 

\end{remark}
Using Lemma \ref{lemma:path-lift}, we can write
\begin{align}
    D^2\Re \Omega(y) = \begin{pmatrix}
    \partial_1^2 \cos(\omega(\arg(y))) & \partial_1 \partial_2 \cos(\omega(\arg(y))) \\
    \partial_2 \partial_1 \cos(\omega(\arg(y))) & \partial_2^2 \cos(\omega(\arg(y)))
    \end{pmatrix}
\end{align}
The same holds for $D^2\Im \Omega(y)$ when we replace $\cos$ with $\sin$. We compute
\begin{align*}
    \partial_1^2 \left( \cos(\omega(\arg(y))) \right) & = \left((D^2 \cos)(\omega(\arg(y))) \right)\, \left[ \omega'(\arg(y)) \, \partial_1 \arg(y) \right]^2 \\
    & + (D \cos(\omega(\arg(y))))  \omega''(\arg(y))\left( \partial_1\arg(y)\right)^2 \\
    & + (D \cos(\omega(\arg(y))))  \omega'(\arg(y)) \partial_1^2\arg(y).
\end{align*}
The same holds for $\partial_2^2 \cos(\omega(\arg(y)))$ if we replace  $\partial_1$ with $\partial_2$.
We also have the mixed derivatives
\begin{align*}
    \partial_1 \partial_2 \cos(\omega(\arg(y))) & = \left((D^2 \cos)(\omega(\arg(y))) \right) \left[ \omega'(\arg(y)) \left( \partial_1\arg(y)\right) \left( \partial_2\arg(y) \right) \right] \\
    & + \left( (D \cos) (\omega(\arg(y))) \right) \omega''(\arg(y)) \partial_1 \arg(y) \partial_2\arg(y) \\
    & + \left( (D \cos) (\omega(\arg(y))) \right) \omega'(\arg(y)) \partial_1\partial_2\arg(y).
\end{align*}
Clearly the same calculations hold when computing $D^2 \Im \Omega(y)$: we replace $\cos$ with $\sin$. 
Now we split the matrix $D^2 \Re \Omega(y)$ into three matrices, corresponding to the three terms in the sum above. We denote the first one as $M^\Re_1$, the second one as $M^\Re_2$ and the third one as $M^\Re_3$; explicitly, 
\begin{align*}
    & (M^\Re_1)_{1,1} = \left((D^2 \cos)(\omega(\arg(y))) \right)\, \left[ \omega'(\arg(y)) \, \partial_1 \arg(y) \right]^2 \\
    &   (M^\Re_1)_{1,2} = \left((D^2 \cos)(\omega(\arg(y))) \right) \left[ \omega'(\arg(y)) \left( \partial_1\arg(y)\right) \left( \partial_2\arg(y) \right) \right] = (M^\Re_1)_{2,1} \\
    & (M^\Re_1)_{2,2} = \left((D^2 \cos)(\omega(\arg(y))) \right)\, \left[ \omega'(\arg(y)) \, \partial_2 \arg(y) \right]^2.
\end{align*}
And also
\begin{align*}
    & (M^\Re_2)_{1,1} =(D \cos(\omega(\arg(y))))  \omega''(\arg(y))\left( \partial_1\arg(y)\right)^2 \\
    & (M^\Re_2)_{1,2} = \left( (D \cos) (\omega(\arg(y))) \right) \omega''(\arg(y)) \partial_1 \arg(y) \partial_2\arg(y) = (M^\Re_2)_{2,1} \\
    & (M^\Re_2)_{2,2} = (D \cos(\omega(\arg(y))))  \omega''(\arg(y))\left( \partial_2\arg(y)\right)^2. 
\end{align*}
And also
\begin{align*}
    & (M^\Re_3)_{1,1} = (D \cos(\omega(\arg(y))))  \omega'(\arg(y)) \partial_1^2\arg(y) \\
    & (M^\Re_3)_{1,2} = \left( (D \cos) (\omega(\arg(y))) \right) \omega'(\arg(y)) \partial_1\partial_2\arg(y) = (M^\Re_3)_{2,1}\\
    & (M^\Re_3)_{2,2} = (D \cos(\omega(\arg(y))))  \omega'(\arg(y)) \partial_2^2\arg(y).
\end{align*}
We will denote the corresponding matrices for $\Im \Omega(y)$ by $M^\Im_i$, $i=1,2,3$. 
Thus we may split the integral in \eqref{eq:GK_A3_3} as 
\begin{align*}
    & \sum_{i=1}^3 \frac{1}{r} \int \frac{ \left< \begin{pmatrix} 
    x^T M^\Re_i x \\
    x^T M^\Im_i x 
    \end{pmatrix}
    , \nu \right>
    }{r} 
    \phi \ps{\av{\frac{y}{r}}^2} \, d\mu(y) \\
    & =:A_{3,1} + A_{3,2} + A_{3,3}.
\end{align*}

\underline{Estimates for $A_{3,1}$.}
Let us introduce the following notation.
\begin{align*}
    & [11]_1 := \left[ \omega'(\arg(y)) \, \partial_1 \arg(y) \right]^2 \\
    & [12]_1 := \left[ \omega'(\arg(y)) \left( \partial_1\arg(y)\right) \left( \partial_2\arg(y) \right) \right] =: [21]_1\\
    & [22]_1 := \left[ \omega'(\arg(y)) \, \partial_2 \arg(y) \right]^2.
\end{align*}
Then we can write 
\begin{align*}
    M^\Re_1 = \begin{pmatrix}
    D^2\cos(\hdots)[11]_1 & D^2 \cos( \hdots)[12]_1 \\
    D^2 \cos(\hdots) [21]_1 & D^2 \cos( \hdots) [22]_1
    \end{pmatrix}
    ,
\end{align*}
where the dots stand for $\omega(\arg(y))$. 
We compute:
\begin{align} \label{e:A3-III-A31a}
  \left< 
  \begin{pmatrix}
  x^T M^\Re_1 x \\
  x^T M^\Im_1 x 
  \end{pmatrix}
  , \nu \right> & = -x_1^2 [11]_1 \left( \cos(\hdots) \nu_1 + \sin(\hdots) \nu_2 \right) \nonumber \\
  & -2x_1 x_2 [12]_1 \left( \cos(\hdots) \nu_1 + \sin(\hdots) \nu_2 \right) \nonumber \\
  & - x_2^2 [22]_1 \left( \cos(\hdots) \nu_1 + \sin(\hdots) \nu_2\right)\nonumber\\
  & = - \left< \Omega(y), \nu \right> \left( x_1^2[11]_1 + 2x_1x_2 [12]_1 + x_2^2 [22]_1\right).
\end{align}
Now, because
\begin{align*}
    \left| x_1^2[11]_1 + 2x_1x_2 [12]_1 + x_2^2 [22]_1\right|  \lesssim \frac{|x|^2}{|y|^2},
\end{align*}
we see that
\begin{align*}
    \av{\frac{1}{r} \int |y| \frac{ \left< \begin{pmatrix} 
    x^T M^\Re_1 x \\
    x^T M^\Im_1 x 
    \end{pmatrix}
    , \nu \right>
    }{r} 
    \phi \ps{\av{\frac{y}{r}}^2} \, d\mu(y)}
    \lesssim \frac{|x|^2}{r^4} \int_{B(0, r)} \dist(y, L_Q) \, d\mu(y),
\end{align*}
where we used again that $\dist(\xi, L_Q) \leq \dist(y, L_Q)$.  

\underline{Estimates for $A_{3,2}$.} Let us write
\begin{align*}
    & [11]_2 :=\omega''(\arg(y)) \left( \partial_1 \arg(y)\right)^2 \\
    & [12]_2 :=  \omega''(\arg(y))  \partial_1 \arg(y) \partial_2\arg(y)=:[21]_2 \\
    & [22]_2 :=  \omega''(\arg(y)) (\partial_2 \arg(y))^2.
\end{align*}
We have
\begin{align*}
    M^\Re_2 = \begin{pmatrix}
    D \cos( \hdots) [11]_2 & D \cos( \hdots) [12]_2 \\
    D \cos(\hdots) [21]_2 & D \cos (\hdots) [22]_2
    \end{pmatrix}.
\end{align*}
As before, we write out $\left< (x^T M^\Re_2 x, x^T M^\Im_2 x), \nu \right>$ to see that it equals to
\begin{align*}
    & x_1^2 [11]_2  \left( D \cos(\hdots) \nu_1 + D \sin (\hdots)  \nu_2 \right) \\
   &  +  x_2^2 [22]_2 \left( D \cos(\hdots) \nu_1 + D \sin(\hdots)\nu_1\right) \\
   & + 2 x_1 x_2 [12]_2 \left( D \cos (\hdots) \nu_1 + D \sin(\hdots) \nu_2 \right).
\end{align*}
Note that
\begin{align} \label{e:A3-a}
    [11]_2= \omega''(\arg(y)) (\partial_1 \arg(y))^2 = \omega'(\arg(y)) \frac{\omega''(\arg(y))}{\omega'(\arg(y))} \partial_1(\arg(y))^2,
\end{align}
and similarly for $[12]_2$, $[21]_2$ and $[22]_2$. 
For all $y \in A(0, r/2, 2r) \cap \spt(\mu)$, we see that 
\begin{align*}
\frac{\omega''(\arg(y))}{\omega'(\arg(y))} \leq C,
\end{align*}
for some constant $C=C(\omega)>0$ depending on $\omega$, and thus on $\Omega$. Indeed, because $\omega$ is bi-Lipschitz, we have that $\omega'(y) \geq (1+\delta_\Omega)^{-1}$. Moreover, $\omega''$ is continuous, since we assumed $\Omega$ to be twice continously differentiable. Because $y \in A(0, r/2, 2) \cap \spt(\mu)$ and using \eqref{eq:SYM_betasmall}, we see that $\arg(y)$ is contained in a compact subset of $[0,2\pi)$, and thus $\omega''(y)$ has a maximum.
Hence we see that
\begin{align*}
    & x^2_1[11]_2 (D\cos(\hdots)\nu_1 + D\sin(\hdots)\nu_2) \\
    &= x^2_1 \omega'(\arg(y)) \frac{\omega''(\arg(y))}{\omega'(\arg(y))} \partial_1(\arg(y))^2 \ps{ D\cos(\hdots) \nu_1 + D\sin(\hdots) \nu_2 } \\
    & = x^2_1\frac{\omega''(\arg(y))}{\omega'(\arg(y))} \partial_1\arg(y) \\& \enskip \enskip \times \ps{\left[ D\cos(\hdots) \omega'(\arg(y)) \partial_1(\arg(y))\right]\nu_1 + \left[ D\cos(\hdots) \omega'(\arg(y)) \partial_2(\arg(y))\right]\nu_2 } \\
    & = x^2_1\frac{\omega''(\arg(y))}{\omega'(\arg(y))} \partial_1\arg(y) \ps{\partial_1 \left[\cos(\hdots)\right]\nu_1 + \partial_1\left[\sin(\hdots)\right]\nu_2}
\end{align*}
Then, expanding out the inner product, re-arranging the terms, using \eqref{e:A3-a} and following the calculations above with the other terms $[12]_2$, $[21]_2$ and $[22]_2$, we see that
\begin{align*}
    \left< \left( x^T M^\Re_2 x, x^T M^\Im_2 x \right) , \nu \right> & = x_1 \frac{\omega''(\arg y)}{\omega'(\arg y)} \partial_1 \arg(y) \, x_1 \left( \partial_1 \left[\cos(\hdots)\right] \nu_1 + \partial_1 \left[\sin(\hdots)\right] \nu_2 \right) \\
    & + x_2 \frac{\omega''(\arg y)}{\omega'(\arg y)} \partial_2 \arg(y) \, x_2 \left( \partial_2 \left[ \cos(\hdots)\right] \nu_1 + \partial_2 \left[\sin(\hdots)\right] \nu_2 \right) \\
    & + x_1 \frac{\omega''(\arg y)}{\omega'(\arg y)} \partial_1 \arg(y) \, x_2 \left( \partial_2 \left[\cos (\hdots)\right] \nu_1 + \partial_2 \left[\sin(\hdots)\right] \, \nu_2 \right) \\
    & + x_2 \frac{\omega''(\arg y)}{\omega'(\arg y)} \partial_2 \arg(y) \, x_1 \left( \partial_1 \left[\cos(\hdots)\right] \nu_1 + \partial_1 \left[\sin(\hdots)\right] \nu_2 \right) \\
    & = \left< D\Omega(y)\cdot x', \nu \right> + \left< D \Omega (y) \cdot x'', \nu \right>,
\end{align*}
where 
\begin{align*}
    x' = \begin{pmatrix}
    x_1 \, \left( x_1 \frac{\omega''(\arg y)}{\omega'(\arg y)} \partial_1 \arg(y) \right) \\
    x_2 \, \left(x_2 \frac{\omega''(\arg y)}{\omega'(\arg y)} \partial_2 \arg(y)\right)
    \end{pmatrix}
\end{align*}
and
\begin{align*}
    x'' = 
    \begin{pmatrix}
    x_1 \left(x_2 \frac{\omega''(\arg y)}{\omega'(\arg y)} \partial_2 \arg(y)\right) \\
     x_2 \left( x_1 \frac{\omega''(\arg y)}{\omega'(\arg y)} \partial_1 \arg(y)\right)
     \end{pmatrix}.
\end{align*}
We see first that $\partial_2 \arg y = \frac{y_1}{|y|^2}$ and that $\partial_1 \arg y= \frac{-y_2}{|y|^2}$. Thus we have that 
\begin{align*}
    (x')_1 = - x_1 x_1 (-y_2 )\frac{\omega''(\arg y)}{|y|^2 \omega'(\arg y)},
\end{align*}
and similarly 
\begin{align*}
   (x')_2 = x_2 x_2 x_1 \frac{\omega''(\arg y)}{|y|^2 \omega'(\arg y)}. 
\end{align*}
By Taylor's theorem, we have that
\begin{align} \label{e:A3-2}
    & y_1 = \frac{|y|}{|x|}x_1 - |y| \sin(\zeta_1)( \theta_y - \theta_x)  \nonumber\\
    & y_2 = \frac{|y|}{|x|}x_2 + |y| \cos(\zeta_2)(\theta_y - \theta_x),
\end{align}
where $\zeta_1, \zeta_2 \in [\theta_y, \theta_x]$ and $\theta_y$ (resp. $\theta_x$) is the angle that $y$ (resp. $x$) makes with the real axis. 
Hence we may write 
\begin{align*}
    x' = \frac{1}{|y|^2}\begin{pmatrix}
    x_1 \left(x_1 (-x_2) \frac{|y|}{|x|} \frac{\omega''(\arg y)}{\omega'(\arg y)} \right)\\
    x_2 \left( x_2 x_1 \frac{|y|}{|x|} \frac{\omega''(\arg y)}{\omega'(\arg y)} \right)
    \end{pmatrix}
    + 
    \frac{1}{|y|^2}\begin{pmatrix}
    x_1 \left(x_1 |y| \sin(\zeta_1) (\theta_y-\theta_x)\frac{\omega''(\arg y)}{\omega'(\arg y)} \right) \\
    x_2 \left(x_2 |y| \cos(\zeta_2) (\theta_y-\theta_x)\frac{\omega''(\arg y)}{\omega'(\arg y)} \right)
    \end{pmatrix}.
\end{align*}
And also we have
\begin{align*}
    x'' = \frac{1}{|y|^2}\begin{pmatrix}
    x_1 \left( x_2 x_1 \frac{|y|}{|x|} \frac{\omega''(\arg y)}{\omega'(\arg y)} \right) \\
    x_2 \left( x_1 (-x_2) \frac{|y|}{|x|} \frac{\omega''(\arg y)}{\omega'(\arg y)} \right)
    \end{pmatrix}
    + 
   \frac{1}{|y|^2} \begin{pmatrix}
    x_1 \left(x_2 |y| \cos(\zeta_1) (\theta_y-\theta_x)\frac{\omega''(\arg y)}{\omega'(\arg y)} \right) \\
    - x_2 \left(x_2 |y| \sin(\zeta_2) (\theta_y - \theta_x)\frac{\omega''(\arg y)}{\omega'(\arg y)} \right)
    \end{pmatrix}
\end{align*}
Note that the first component of $x'$ and the first component of $x''$ cancel each other out. Moreover, we see that
\begin{align} \label{e:A3-III-A32-a}
    |y| |\theta_y- \theta_x| \lesssim \dist(y, L_Q).
\end{align}
Indeed, we because of the assumption \eqref{eq:SYM_betasmall}, we may suppose that $|\theta_y - \theta_x| \leq \frac{\pi}{4}$. Hence we have 
$|\theta_y-\theta_x|^2 \sim \sin(|\theta_y-\theta_x|)^2 + \sin(|\theta_y-\theta_x|/2)^4 \leq 2 \sin(|\theta_y-\theta_x|)^2$.
Recall that $|y| |D\Omega(y)| \lesssim 1$. Thus, recalling also that $|D\Omega(y)| \lesssim |y|^{-1}$, 
\begin{align} \label{e:A3-III-A32-b}
    & \av{ \frac{1}{r} \int |y|\frac{\left< \left( x^T M^\Re_2 x, x^T M^\Im_2 x \right) , \nu \right>}{r} \phi \ps{\av{\frac{y}{r}}^2} \, d\mu(y) } \nonumber \\
    & \lesssim \frac{1}{r^2} \int |y|\frac{\omega''(\arg y)}{|y|^2 \omega'(\arg y)} |D \Omega(y)| |x|^2 \dist(y, L_Q) \phi \ps{\av{\frac{y}{r}}^2} \, d\mu(y) \nonumber\\
    & \leq C \frac{|x|^2}{r^4} \int_{B(0,r)} \dist(y, L_Q)\, d\mu(y).
\end{align}

\underline{Estimates for $A_{3,3}$}.
We see that we may re write $\left< \left(x^T M_3^\Re x, x^T M^\Im_3 x\right), \nu \right>$ as 
\begin{align*}
    & \left( D \cos (\hdots) \nu_1 + D \sin(\hdots) \nu_2 \right) \omega'(\arg y) |y|^{-4} \\
    & \times \left[ (x_1^2 (2 y_1 y_2) - 2 x_1 x_2 y^2_1) + (x_2^2 (-2y_1 y_2) + 2x_1 x_2 y_2^2\right] \\
    & =: \left( \hdots \right) \times [\mathcal{A} + \mathcal{B}].
\end{align*}
If we expand as above $y_1$ and $y_2$ in terms of $x_1 $ and $x_2$ respectively, we obtain that
\begin{align*}
    \mathcal{A} & = x_1^2 \frac{|y|^2}{|x|} x_2 \cos(\zeta)(\theta_x - \theta_y) - x_1^2 \frac{|y|^2}{|x|} x_2 \sin(\zeta)(\theta_x - \theta_y) - x_1^2 |y|^2 \sin(\xi) \cos(\zeta) (\theta_x - \theta_y)^2\\
    & + 4 x_1 x_2 \frac{|y|^2}{|x|}x_1 \sin(\zeta) (\theta_x - \theta_y) - 2 x_1 x_2 |y|^2 \sin(\zeta)^2 (\theta_x - \theta_y)^2 =: \sum_{j=1}^5 \mathcal{A}_i.
\end{align*}
It is immediate to see that $\mathcal{B}$ has the very same structure. Notice that for $i=1,2,4$ we have that $|\mathcal{A}_i|, |\mathcal{B}_i| \lesssim |y||x|^2 \dist(y, L_Q)$; on the other hand, for $i=3,5$, we have that $|\mathcal{A}_i|,|\mathcal{B}_i|\lesssim |x|^2 \dist(y,L_Q)^2 \lesssim |y||x|^2 \dist(y, L_Q)$.
Thus 
\begin{align*}
    & \frac{1}{r^2}\int_{B(0,r)} |y|\av{\left( D \cos (\hdots) \nu_1 + D \sin(\hdots) \nu_2 \right) \omega'(\arg y) |y|^{-4} }|y||x|^2 \dist(y, L_Q) \, d \mu(y) \\
     & \lesssim \frac{|x|^2}{r^4} \int_{B(0,r)} \dist(y, L_Q) \, d\mu(y).
\end{align*}
This shows the desired bound for $A_{3,3}$, and thus gives a proof of Lemma \ref{l:III}.

\end{document}